\documentclass{amsart}
\usepackage[T1]{fontenc}
\usepackage{graphicx,amsmath,amsfonts,amssymb,color,amsthm} 
\usepackage[bookmarks=true,%
    colorlinks=true,%
    linkcolor=blue,
    citecolor=blue,%
    filecolor=blue,%
    menucolor=blue,%
    urlcolor=blue,%
    breaklinks=true]{hyperref}
\usepackage[all]{xy}
\usepackage{tikz-cd}
\usepackage{enumitem}
\newcommand{\Q}{\mathbb{Q}}
\newcommand{\Z}{\mathbb{Z}}
\newcommand{\C}{\mathbb{C}}
\newcommand{\R}{\mathbb{R}}
\newcommand{\SL}{\mathrm{SL}}

\renewcommand{\sl}{\mathrm{sl}}
\newcommand{\calC}{\mathcal{C}}
\newcommand{\calV}{\mathcal{V}}
\newcommand{\calX}{\mathcal{X}}
\newcommand{\calN}{\mathcal{N}}
\newcommand{\calM}{\mathcal{M}}
\newcommand{\Mod}{\mathrm{Mod}}
\newcommand{\PMod}{\mathrm{PMod}}

\newcommand{\SU}{\mathrm{SU}}
\newcommand{\SO}{\mathrm{SO}}

\newcommand{\calO}{\mathcal{O}}

\renewcommand{\S}{\mathrm{S}}

\DeclareMathOperator{\End}{End}

\DeclareMathOperator{\Tr}{Tr}

\DeclareMathOperator{\Hom}{Hom}
\DeclareMathOperator{\Sign}{Sign}

\DeclareMathOperator{\Res}{Res}

\def\calP{\mathcal{P}}
\def\calQ{\mathcal{Q}}

\newtheorem{Theorem}{Theorem}[section]
\newtheorem{theoremA}{Theorem}

\theoremstyle{definition}
\newtheorem{Lemma}[Theorem]{Lemma}
\newtheorem{Definition}[Theorem]{Definition}
\newtheorem{Proposition}[Theorem]{Proposition}
\newtheorem{Remark}[Theorem]{Remark}

\title{Reidemeister torsion of two-bridge knots and signatures of TQFT}
\author{Julien March\'{e}}
\author{Seokbeom Yoon}
\date{\today}

\begin{document}
\begin{abstract}
We establish an explicit relation between the adjoint Reidemeister torsion of the two-bridge knot $K(p,q)$ at any parabolic representation and the Frobenius algebra governing the signatures of SU$_2$-TQFT vector spaces at the root $\zeta=\exp(i\pi q/p)$. As applications, (a) we prove that the inverse sum of torsions is constant (i.e., independent of $p$ and $q$); and (b) we show that along sequences of roots of the form 
$\zeta_n = \exp\left(i\pi\tfrac{a+bn}{c+dn}\right)$, 
the signatures have the same asymptotic behavior as the Verlinde formula. 
\end{abstract}
\maketitle

\tableofcontents

\section{Introduction}

Let $p$ and $q$ be coprime odd integers satisfying $0<q<p$ and denote by $K(p,q)$ the two-bridge knot with these parameters. Independently, setting $\zeta=\exp(i \pi q/p)$, one can consider the $\SU_2$-TQFT with defining root $\zeta$. For any genus $g$, it gives a Hermitian vector space $\calV_\zeta(S_g)$ endowed with a projective unitary representation of the mapping class group $\Mod(S_g)$. This article establishes a tight relation between, on the one hand, the space of parabolic $\SL_2(\C)$-representations of $\pi_1(S^3\setminus K(p,q))$ and the adjoint Reidemeister torsion, and, on the other hand, a Frobenius algebra governing the signatures of the various vector spaces $\calV_\zeta(S_g)$. It completes the work initiated in \cite{S2P} and provides applications both to the computation of the Reidemeister torsion itself and to the study of the asymptotic behavior of the signatures for a sequence $\zeta_n$ of roots converging to a given root $\zeta$. 

In both the introduction and the main part of the paper, we have tried to separate, as much as possible, the knot-theoretic aspects from the TQFT aspects for the reader’s benefit. Before proceeding further, we caution that our proofs are computational in nature: we do not yet have a conceptual explanation for the relations we uncover. Finding such an explanation remains an interesting open problem.

\subsection{About two-bridge knots}

We denote by $\calX$ the set of conjugacy classes of irreducible representations $\rho:\pi_1(S^3\setminus K(p,q))\to \SL_2(\C)$ which send any (hence every) meridian to a non-trivial matrix with trace $2$. This set is a finite affine algebraic set defined over $\Z$, hence its coordinate ring $\Q[\calX]$ is a finite commutative $\Q$-algebra. Riley showed in \cite{Riley} that it is semi-simple. 

Given $[\rho]\in \calX$ and any $n\ge 1$, we define $\rho_n$ to be the composition of $\rho$ with the $(n+1)$-dimensional irreducible representation of $\SL_2(\C)$. We refer to $\rho=\rho_1$ as the tautological representation and to $\rho_2$ as the adjoint representation. 
By choosing appropriate generators of the twisted cohomology groups $\mathrm{H}^*(S^3\setminus K(p,q),\rho_n)$, one can define the Reidemeister torsion $\tau_n(\rho)$ as a complex number depending on $[\rho] \in \calX$. By construction, the map $\rho\mapsto \tau_n(\rho)$ is conjugation-invariant and algebraic, yielding an element $\tau_n \in \Q[\calX]$. Throughout the paper, we consider $\tau_n$ up to sign to avoid unnecessary complications. 

It will be technically more convenient to replace $\calX$ with a ramified 2-cover $\hat{\calX}$, and hence $\Q[\calX]$ with the corresponding quadratic extension $\Q[\hat{\calX}]$. Topologically, this amounts to recalling that $K(p,q)$ is doubly amphichiral, meaning that there exist two commuting involutions of $S^3$ which preserve $K(p,q)$ while reversing its orientation. Using this fact, we construct an orbifold quotient of $S^3\setminus K(p,q)$ whose orbifold fundamental group $G$ fits in the (split) exact sequence 
$$0\to \pi_1(S^3\setminus K(p,q))\to G\to Q_8\to 0$$
where $Q_8$ is the usual quaternion group of order $8$. Note that similar groups have already been considered in \cite{LS}. Then $\hat{\calX}$ is the set of conjugacy classes of representations $\rho:G\to \SL_2(\C)$ whose restriction to $\pi_1(S^3\setminus K(p,q))$ defines an element of $\calX$. 

Concretely, we will show that there is an isomorphism $\Q[\hat{\calX}] \cong \Q[X]/(P_{p-1}(X))$ where $P_k \in \Z[X]$ is the sequence of polynomials given by 
$$P_0=1,\ P_1=X\text{ and }P_{k}=\epsilon_k XP_{k-1}+P_{k-2}\text{ where }\epsilon_k=(-1)^{\lfloor \frac{kq}{p}\rfloor}.$$
We set $V=\Q[\hat{\calX}]$ and $V^+=\Q[\calX]$ to simplify notation, and denote by $x$ and $\calP_k$ the image of $X$ and $P_k$, respectively, in $V$. It turns out that $V^+$ is spanned by $\calP_k$ for even $0 \leq k \leq p-3$, and that $V$ is the quadratic extension of $V^+$ obtained by adjoining $\calP_{p-2}$, which is a square root of $-1$ in $V$.

The recurrence relation defining the sequence $P_k$ is typical for orthogonal polynomials. Indeed, there exists a non-degenerate bilinear form $\eta:V\times V\to \Q$ for which the basis $\calP_0,\ldots,\calP_{p-2}$ is orthogonal. It is given by $\eta(x,y)=\epsilon(xy)$ where $\epsilon:V\to \Q$ is the linear form satisfying $\epsilon(\calP_0)=1$ and $\epsilon(\calP_k)=0$ for $0<k<p-1$. 
The data of $\eta$, or equivalently of $\epsilon$, endows $V$ with the structure of a commutative semi-simple Frobenius algebra. 
At first glance, this appears unrelated to the topological definition of $V$, however, one of our main theorems is that this additional structure arises from the adjoint Reidemeister torsion.
To explain this, recall that for any Frobenius algebra $(V,\eta)$, one defines an element $\Omega=\sum x_iy_i$ where $\eta^{-1}=\sum x_i\otimes y_i\in V\otimes V$. When $V$ is semi-simple, $\epsilon$ can be recovered from $\Omega$ by the formula $\epsilon(x)=\Tr_{V/\Q}(\Omega^{-1}x)$.

\begin{theoremA} [Theorems~\ref{thm.simp}, \ref{toradj} and Equation~\eqref{toromega}]
    Let $0<q<p$ be coprime odd integers and $V$ be the Frobenius algebra described above with associated element $\Omega$. Denoting by $\tau_n$ the Reidemeister torsion of $S^3\setminus K(p,q)$ with coefficients $\rho_n$, one has
    $$\tau_1=\frac{2}{x^2\calP_{\ell-1}}\quad \text{and} \quad \tau_2=\frac{\Omega}{4x^2\calP_{\ell-1}^2}$$
    where $0<\ell<p$ is the unique odd integer satisfying $q\ell \equiv \pm 1 \pmod{p}$.
\end{theoremA}

The appearance of $\calP_{\ell-1}$ deserves some explanation. By the classification of two-bridge knots, $K(p,q)$ and $K(p,\ell)$ are isotopic. Denoting by $\calX^*,V^*,\ldots$ all the previous data associated to $K(p,\ell)$, it follows that $\hat{\calX}$ and $\hat{\calX}^*$ are isomorphic as algebraic varieties, hence $V$ and $V^*$ are isomorphic as algebras. More precisely, we will show that the isomorphism $V^*\leftrightarrow V$ sends $x^*$ to $x\calP_{\ell-1}$ and $x$ to $x^*\calP^*_{q-1}$. 
As $\tau_n$ is a topological invariant, the formulas provided in $V$ and $V^*$ should match. This is obvious for $\tau_1$, if one writes it as $\tau_1=\frac{2}{xx^*}$. In the case of $\tau_2$, we have instead:
$$\tau_2=\frac{\Omega}{4x(x^*)^2}=\frac{\Omega^*}{4x^*x^2}.$$
This yields the reciprocity formula $x\Omega=x^*\Omega^*$, showing that the Frobenius algebra structures on $V$ and $V^*$ are indeed distinct. In addition, this formula will play a crucial role in our study of the asymptotics of the signature of TQFT.

As a first application of our formulas, we will prove the following two equations, where the latter extends the result of \cite{Yoon} to the parabolic case and has a quantum-physical motivation; see \cite{GKY}. 
Note that the condition $q \neq 1$ in the following theorem implies that the two-bridge knot $K(p,q)$ is hyperbolic, and that the theorem fails when $q=1$.

\begin{theoremA}[Theorems~\ref{thm.invtau2} and \ref{thm.invtau1}]
Let $0<q<p$ be coprime odd integers with $q\neq1$ and $\calX$ be the parabolic character variety of $S^3\setminus K(p,q)$. Then one has 
$$\sum_{\rho\in \calX}\frac{1}{\tau_1(\rho)}=\pm 1 \quad \text{and} \quad  \sum_{\rho\in \calX}\frac{1}{\tau_2(\rho)}=0.$$
\end{theoremA}

\subsection{About signatures of TQFT}

Fix an odd integer $p$ and let $\Q(\zeta)$ be the cyclotomic field of order $2p$, where $\zeta$ is a primitive $2p$-th root of unity. We set $\Lambda=\{0,1,\ldots,p-2\}$ and $\Lambda^+=\{0,2,\ldots,p-3\}$ which we use as ``colors'' indexing marked points. In this article, we consider the $\SU_2/\SO_3$-TQFT as a collection of $\Q(\zeta)$-vector spaces $\calV_p(S_{g,n};\lambda)$ where $S_{g,n}$ is a surface of genus $g$ with $n$ marked points and $\lambda\in \Lambda^n$ is a color of the marked points. These are endowed with an Hermitian form defined over $\Q(\zeta)$ and an action of the pure mapping class group $\PMod(S_{g,n})$ by projective isometries. 
A key property is that these vector spaces satisfy compatibility conditions under gluing. If $S$ is a surface and $\gamma\subset S$ is a simple closed curve disjoint from the marked points, then a new (possibly disconnected) surface $S'$ is obtained by cutting $S$ along $\gamma$, capping off the resulting boundaries with discs, and adding a new marked point in each disc. When $\gamma$ is non-separating, the compatibility means an isomorphism 
$$\calV_p(S;\lambda)=\bigoplus_{\mu\in \Lambda} \calV_p(S';\lambda,\mu,\mu)$$
which is compatible with both the Hermitian structure and the mapping class group action. There is a similar statement for the separating case. 

In this article, we embed $\Q(\zeta)$ into $\C$ by setting $\zeta=\exp(i\pi \tfrac{q}{p})$ and define $\C$-vector spaces as $\calV_\zeta(S_{g,n};\lambda)=\calV_p(S_{g,n};\lambda)\otimes \C$ whose signature is denoted by
$$\sigma_g (\tfrac{q}{p};\lambda)=\Sign \calV_\zeta(S_{g,n};\lambda).$$ 
Since any surface can be decomposed into pairs of pants, the aforementioned compatibility allows one to reduce the computation of the signature to the sphere with at most three marked points, by introducing the so-called signed Verlinde Frobenius algebra \cite{DM}. It is a $(p-1)$-dimensional vector space $V=\bigoplus_{k\in \Lambda} \Q e_k$ equipped with a symmetric bilinear form $\eta:V\times V\to \Q$ and a trilinear symmetric form $\omega:V\times V\times V\to \Q$, given by 
\begin{align*}
\eta(e_j,e_k)&=\sigma_0(\tfrac{q}{p};\,j,k)=\delta_{jk}(-1)^j\epsilon_{j+1} \,,\\
\omega(e_j,e_k,e_l)&=\sigma_0(\tfrac{q}{p};\,j,k,l)\in \{-1,0,1\}\,.
\end{align*}
Here $\delta_{jk}$ denotes the Kronecker delta.
It was shown in \cite{DM} that there is a unique structure of semi-simple algebra on $V$ so that one has $\omega(x,y,z)=\eta(xy,z)$ for any $x,y,z\in V$. Moreover, it was shown in \cite{S2P} that sending $e_k$ to $\calP_k$ defines an algebra isomorphism from $V$ to $\Q[\hat{\calX}]$. This is why we have already denoted it by $V$ in the first part of the introduction. The initial motivation of this article was to extend the isomorphism $V\simeq \Q[\hat{\calX}]$ to an isomorphism of Frobenius algebras. To achieve this, one needs to give a topological meaning to the element $\Omega$. It was a pleasant surprise that the adjoint Reidemeister torsion turns out to play exactly this role.

One can recovers the signatures from the Frobenius algebra $V$ by means of the following formula, which holds for any surface $S_{g,n}$ colored by $\lambda=(\lambda_1,\ldots,\lambda_n)$:
$$ \sigma_g(\tfrac{q}{p};\lambda)=\epsilon(\Omega^g e_{\lambda_1}\cdots e_{\lambda_n})=\Tr_{V/\Q}(\Omega^{g-1}e_{\lambda_1}\cdots e_{\lambda_n}).$$
There is a similar $\SO_3$-TQFT for which the set of colors is restricted to $\Lambda^+$. The corresponding Frobenius algebra is the subalgebra $V^+\subset V$ consisting of even colors. As $\Omega\in V^+$, this subalgebra naturally inherits a Frobenius algebra structure. One often prefers the $\SO_3$-TQFT because it has smaller dimension and better properties, such as irreducibility and the existence of an integral basis. In this article, however, we restrict to the $\SU_2$-case, as their signatures only differ by a power of $2$.  

\subsection{Asymptotics of the signature along roots of unity}

Let us first recall that the dimension of $\calV_p(S_g)$ is given by the famous Verlinde formula (we set $n=0$, i.e., we consider surfaces with no marked points from now on):
$$ \dim \calV_p(S_g)=\left(\frac{p}{2}\right)^{g-1}\sum_{k=1}^{p-1}\sin \big(\frac{k\pi}{p}\big)^{2-2g}.$$
One can easily deduce from the formula the asymptotic behavior 
$$\dim \calV_p(S_g)\underset{p\to\infty}{\sim} 2\left(\frac{p^3}{2\pi^2}\right)^{g-1} \zeta(2g-2),$$
which admits a beautiful interpretation; see, e.g., \cite{Witten}.
For motivational purposes, we briefly recall this. Regarding the vector space $\calV_p(S_g)$ as the geometric quantization of the character variety $$\calM(S_g)=\Hom(\pi_1(S_g),\SU_2)/\SU_2,$$ one can endow $\calM(S_g)$ with a complex structure and construct a holomorphic line bundle $L$ over $\calM(S_g)$ such that 
$$ \calV_p(S_g)\simeq \mathrm{H}^0(\calM(S_g),L^{p-2}).$$
Since $c_1(L)$ represents the Atiyah-Bott-Goldman symplectic (Kähler) structure on $\calM(S_g)$, the higher cohomology groups vanish for sufficiently large $p$, and the Hirzebruch-Riemann-Roch theorem yields
$$\dim \calV_p(S_g)=\int_{\calM(S_g)}e^{p \, c_1(L)}\operatorname{Todd}(X)\underset{p\to\infty}{\sim}p^{3g-3}\operatorname{Vol}(\calM(S_g)),$$
together with the non-trivial fact that the Verlinde formula is a polynomial in $p$.

It is then natural to ask whether a similar behavior occurs for the signature along sequences of roots of unity. Numerical experiments show that the answer depends strongly on the specific sequence of roots of unity. Our final main theorem provides an explicit asymptotic formula for the signatures.

\begin{theoremA}[Theorem~\ref{mainsign}]
Let $q_n=a+bn$ and $p_n=c+dn$, where $n$ is odd and $a,b,c,d \in \Z$ satisfy $ad-bc=1$ with $a,d$  odd, $b,c$ even, $0\le b<d$ and $c>0$.
Then for any $g\ge 2$, 
\begin{enumerate}
    \item the signature $\sigma_g(\frac{q_n}{p_n})$ is a polynomial in $n$, and
    \item under a technical condition \eqref{cond:H}, there exists a semi-simple $d$-dimensional Frobenius algebra $(W,\Omega_W)$ depending on $a,b,c,d$ such that 
$$\sigma_g \left(\frac{q_n}{p_n}\right)\underset{n\to\infty}{\sim}2\left(\frac{n^3}{2\pi^2}\right)^{g-1} \! \! \zeta(2g-2)\Tr_{W/\Q}(\Omega_W^{g-1}).$$
\end{enumerate}
\end{theoremA}

In the introduction, we do not describe the Frobenius algebra $(W,\Omega_W)$ or the technical condition \eqref{cond:H}; details can be found in Section~\ref{section:asympto}. We expect that \eqref{cond:H} holds for all $a,b,c,d$ as in the theorem, and we have verified it for all $0\le b<d<100$. For example, when $(a,b,c,d)=(3,2,4,3)$, one finds 
$$W=\Q[X]/(4X^3 + 16X^2 + 23X + 12)\text{ and }\Omega_W=28 X^2+80X+65$$
so that 
\begin{center}
\renewcommand{\arraystretch}{1.3}
\begin{tabular}{|c|cccccc|}
\hline
Genus $g$ & 1 & 2 & 3 & 4 & 5 & 6 \\ 
\hline
$\displaystyle\lim_{n\to\infty} 
\tfrac{\sigma_g(q_n/p_n)}{\dim \calV_{p_n}(S_g)}$
& $\tfrac{1}{3^0}$ & $\tfrac{1}{3^3}$ & $\tfrac{1}{3^6}$ 
& $\tfrac{1345}{3^9}$ & $\tfrac{1762}{3^{12}}$ & $\tfrac{2241}{3^{15}}$ \\
\hline
\end{tabular}
\end{center}
\smallskip

The proof of the above theorem is rather involved, so we outline its main steps for the reader’s convenience.
\begin{enumerate}
\item[Step 1:] One has $\sigma_g(q_n/p_n)=\Tr_{V_n^*/\Q}(\Omega_n^{g-1})$ where $V_n^*$ is the Frobenius algebra associated with $q_n/p_n$. Recall that $V_n^* \simeq V_n$ as algebras, hence the trace can be computed in $V_n$ instead. Then, by construction, the pair $(p_n,q_n)$ is reciprocal to the pair $(p_n,d)$. It will be much more convenient to work in $V_n$, as $q_n$ is replaced by the constant $d$. 
\item[Step 2:] We will show that in $V_n$, the polynomial $P_{p_n-1}$ (and certain other polynomials used in the computation of the signature) arises as a specializations of a $2$-variable polynomial $Q$; explicitly, one has $$P_{p_n-1}(it+it^{-1})(t-t^{-1})^d=i^{n+1} Q(t,t^n).$$ 
\item[Step 3:] We consider the plane algebraic curve $\calC=\{(u,v)\in (\C^*)^2 \, | \, Q(u,v)=0\}$. The signature can be expressed as the sum of a rational function to the power $g-1$ over a finite set $\calC[n]=\{(u,v)\in \calC \, | \, v=u^n\}$. We then study the distribution of these points as $n\to \infty$ and show that the main contribution to the signature concentrates around the singular points $\{\pm 1\}^2\subset \calC$.
\item[Step 4:] We analyze the branches of $\calC$ at the singular points $(\pm 1,\pm 1)$. Under the assumption~\eqref{cond:H}, exactly $d$ branches pass through each singular point, and the points of $\calC[n]$ accumulate along each branch in a regular way. The sum of these contributions, raised to the power $g-1$, is accounted for in the term $\Tr_{W/\Q}\Omega^{g-1}$. 
\end{enumerate}

To end this section, we advise the reader that there is a parallel article which also study the asymptotic behavior of the signature function $\frac{q}{p}\mapsto \sigma_g(\frac{q}{p})$, but in a different regime. Contrary to the case of this article, the authors consider sequences $\frac{q_n}{p_n}$ converging to a (very) irrational limit. The order of convergence is then $p_n^{2g-2}$ instead of $p_n^{3g-3}$ and the limit is expected to have modular properties, see \cite{MarMas}.

\subsection{Back to two-bridge knots}
To end this introduction, let us explain how our last theorem arises naturally from the knot-theoretical side. We do not attempt to make the argument precise. 

The sequence of two-bridge knots $K(p_n,q_n)$ can be drawn in its Conway normal form starting from the continued expansion of $p_n/q_n$. It is known (see e.g. \cite{Khi}) that there are integers $a_1,\ldots,a_k$ (independent of $n$) such that for any $n$
$$\frac{p_n}{q_n}=a_1+\cfrac{1}{a_2+\cfrac{1}{\cdots+\cfrac{1}{a_k+n-1}}}$$
We observe in Figure~\ref{fig.planardiagram} that this sequence of knots can be obtained by performing $\frac{1}{m}$-surgery on the trivial component $L_2$ of the link $L$, where $m=\frac{n-1}{2}$ for odd $n$. In the figure, boxes indicate twist regions, and $L_2$ is the trivial knot encircling the last twist region. This gives the key intuition for the last theorem. 

\begin{figure}[htpb!]
    \centering
    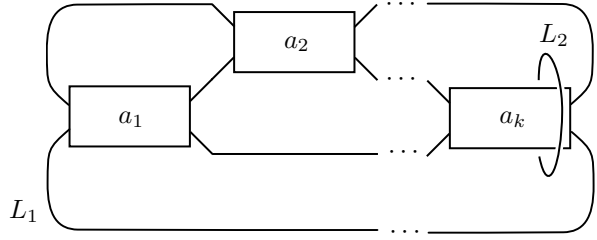
    \caption{The Conway normal form of $[a_1,a_2,\ldots,a_k]$ with a trivial component.}
    \label{fig.planardiagram}
\end{figure}

We consider the character variety $\calX_L$ of representations $\rho:\pi_1(S^3\setminus L)\to \SL_2(\C)$ which send a meridian $m_1$ of the non-trivial component $L_1$ to a non-trivial parabolic matrix. Denoting by $m_2$ and $l_2$ the meridian and longitude of $L_2$, respectively, the knot group of the surgered knot is obtained by adding the relation $m_2l_2^{-m}=1$. 
The character variety $\calX_L$ is expected to be 1-dimensional and to project birationally onto the character variety of the boundary $T_2$ of a tubular neighborhood of $L_2$. The latter variety is parametrized by a pair $(u,v)\in (\C^*)^2$ corresponding to a representation $\rho_{u,v}:\pi_1(T_2)\to \SL_2(\C)$ given by 
$$\rho_{u,v}(l_2)=\begin{pmatrix}u & 0 \\ 0 & u^{-1}\end{pmatrix},\quad \rho_{u,v}(m_2)=\begin{pmatrix}v & 0 \\ 0 & v^{-1}\end{pmatrix}.$$
The image of $\calX_L$ in $(\C^*)^2$ is then a plane curve $\calC$ given by the vanishing of a polynomial analogous to the $A$-polynomial of a knot. The parabolic character variety $\calX_n$ of $K(p_n,q_n)$ maps to the set of points $\calC[m] \subset \calC$ satisfying $v=u^m$.

To end this rough picture, recall that one can define the Reidemeister torsion $\tau_L$ of $S^3\setminus L$ with coefficients $\rho_2$ for $[\rho]\in \calX_L$. This is a meromorphic function on $\calX_L$ which behaves well under surgery, meaning that there is a formula expressing the collection of elements $\Omega_n$ in terms of the function $\tau_L$ evaluated at points in $\calC[m]$. This expression happens to have poles for $u,v=\pm 1$, and the resulting expression for the signature will localize on those poles. 

It is then natural to expect that the Frobenius algebra $W$, explicitly described in Section \ref{section:asympto}, can be described from the parabolic character variety of $L$, i.e., from representations which send both $m_1$ and $m_2$ to parabolic elements. We leave this for future investigation. 

\subsection*{Acknowledgments}
The first author thanks Makoto Sakuma for discussions around the orbifold quotients of two-bridge knot complements and Gregor Masbaum for its constant interest. 

\section{Reidemeister torsions of two-bridge knots}

\subsection{Topological setting}
\label{sec.top}

Let $0<q<p$ be two coprime odd integers and $K(p,q)$ be the two-bridge knot of parameters $(p,q)$. It is characterized by the property that its double branched cover is the lens space $L(p,q)$.
Precisely, the knot $K(p,q)$ is defined as follows. Set
\begin{equation*}
    S^3=\{(z_1,z_2)\in \C^2 \, : \,  |z_1|^2+|z_2|^2=2\} 
\end{equation*}
and let $R:S^3 \rightarrow S^3$ be the homeomorphism defined by 
\begin{equation*}
    R(z_1,z_2) = (\zeta z_1,\zeta^q z_2), \quad \zeta=\exp \left(\frac{2\pi i}{p} \right) \, .
\end{equation*}
The lens space $L(p,q)$ is the quotient of $S^3$ by the free action of $\Z/p\Z$ generated $R$.
The complex conjugation $C(z_1,z_2)=(\overline{z_1},\overline{z_2})$ satisfies $CRC^{-1}=R^{-1}$ and hence extends the $\Z / p \Z$-action to an action of the dihedral group $D_{p}=\Z/p\Z\rtimes\Z/2\Z$.
This $D_p$-action is no longer free, and the resulting quotient is homeomorphic to $S^3$, with ramification locus given by the knot $K(p,q)$.

A standard way to visualize the above construction is to decompose $S^3$ into two solid tori $T_1$ and $T_2$ defined by the inequalities $|z_1|\le |z_2|$ and $|z_2|\le |z_1|$, respectively. Note that their common boundary is a torus given by $|z_1|=|z_2|=1$. The quotient of $T_1$ by $\Z/p\Z$ is also a solid torus
\begin{equation*}
    T_1/ ( \Z / p \Z) = \big\{(r e^{i \theta}, z_2) \, \big |\, 0 \leq r \leq 1 , \ -\frac{\pi}{p} \leq \theta \leq \frac{\pi}{p} ,\  r^2 + |z_2|^2 =2 \big\}  /_\sim
\end{equation*}
where $\sim$ means that the discs bounded by $z_1 = \zeta^\frac{1}{2}$ and $z_1 = \zeta^{-\frac{1}{2}}$ are identified by $R$; see Figure~\ref{fig.lens}. If we further quotient $T_1/(\Z / p \Z)$ by the conjugation $C$, we obtain a ball with some ramification. Precisely, the quotient of $T_1$ by $D_p$ is a ball $B_1$ with ramification locus consisting of two arcs (blue arcs in Figure~\ref{fig.lens}):
\begin{equation}
\label{eqn.ramified}
    \{ (1,t) \, | \, -1 \leq t \leq 1\} \quad \text{and} \quad \{ (\zeta^{\frac{1}{2}},t \zeta^{\frac{q}{2}}) \, | \, -1 \leq t \leq 1\}\,.
\end{equation}
Similarly, the quotient $T_2/D_p$ is also a ball $B_2$ containing two ramified arcs. Gluing two balls $B_1$ and $B_2$ yields the 3-sphere, in which the ramified arcs combine to form the knot $K(p,q)$.

\begin{figure}[htpb!]
    \centering
    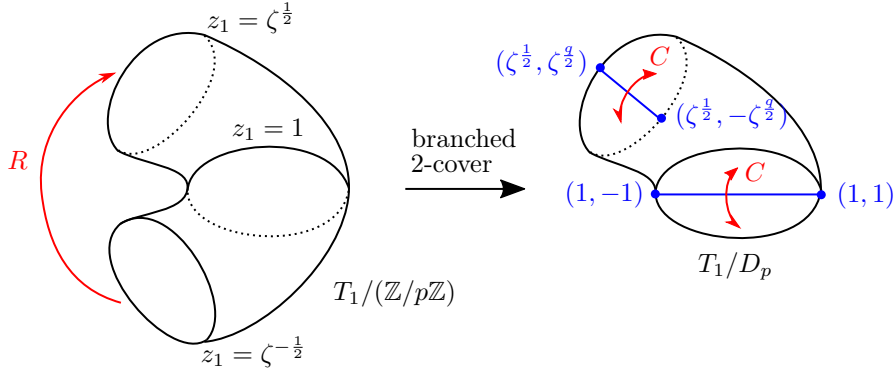
    \caption{The quotients of $T_1$ by $\Z/p \Z$ and $D_p$.}
    \label{fig.lens}
\end{figure}

If we project the ball $B_1$ onto a planar square so that the two ramified arcs in~\eqref{eqn.ramified} are mapped to its vertical edges, the two ramified arcs in $B_2$ are attached to the vertical edges with slope $q/p$, as illustrated in Figure~\ref{fig.schubert}. The resulting diagram is  called Schubert's normal form of $K(p,q)$. 
Letting $u$ and $v$ be the Wirtinger generators associated with the vertical edges, we obtain the following presentation of the knot group $G=\pi_1(S^3\setminus K(p,q))$:
\begin{equation}\label{presentation}
G=\langle u,v \, | \, wuw^{-1}v^{-1} \rangle, \quad w=u^{\epsilon_1}v^{\epsilon_2}\cdots u^{\epsilon_{p-2}} v^{\epsilon_{p-1}}
\end{equation}
where $\epsilon_k=(-1)^{\lfloor \frac{kq}{p}\rfloor}$ for any $k \in \Z$ and the relator in~\eqref{presentation} is obtained by traversing the diagram once.
\begin{figure}[htpb!]
    \centering
\begingroup%
  \makeatletter%
  \providecommand\color[2][]{%
    \errmessage{(Inkscape) Color is used for the text in Inkscape, but the package 'color.sty' is not loaded}%
    \renewcommand\color[2][]{}%
  }%
  \providecommand\transparent[1]{%
    \errmessage{(Inkscape) Transparency is used (non-zero) for the text in Inkscape, but the package 'transparent.sty' is not loaded}%
    \renewcommand\transparent[1]{}%
  }%
  \providecommand\rotatebox[2]{#2}%
  \newcommand*\fsize{\dimexpr\f@size pt\relax}%
  \newcommand*\lineheight[1]{\fontsize{\fsize}{#1\fsize}\selectfont}%
  \ifx\svgwidth\undefined%
    \setlength{\unitlength}{151.46571602bp}%
    \ifx\svgscale\undefined%
      \relax%
    \else%
      \setlength{\unitlength}{\unitlength * \real{\svgscale}}%
    \fi%
  \else%
    \setlength{\unitlength}{\svgwidth}%
  \fi%
  \global\let\svgwidth\undefined%
  \global\let\svgscale\undefined%
  \makeatother%
  \begin{picture}(1,0.9830618)%
    \lineheight{1}%
    \setlength\tabcolsep{0pt}%
    \put(0,0){\includegraphics[width=\unitlength,page=1]{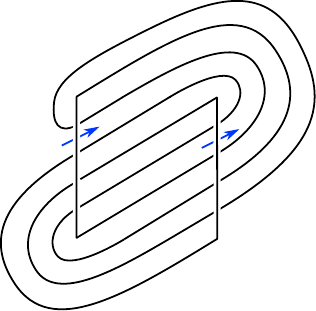}}%
    \put(0.13372162,0.49769947){\color[rgb]{0,0,0}\makebox(0,0)[lt]{\lineheight{1.25}\smash{\begin{tabular}[t]{l}$v$\end{tabular}}}}%
    \put(0.58280719,0.48664854){\color[rgb]{0,0,0}\makebox(0,0)[lt]{\lineheight{1.25}\smash{\begin{tabular}[t]{l}$u$\end{tabular}}}}%
  \end{picture}%
\endgroup%

    \caption{Schubert's normal form of $K(5,3)$.}
    \label{fig.schubert}
\end{figure}

By interchanging the roles of $B_1$ and $B_2$, we obtain another presentation of $G$, which coincides with the presentation \eqref{presentation} of $\pi_1(S^3\setminus K(p,\ell))$. Here $0<\ell<p$ is the unique odd integer satisfying $q\ell \equiv \pm 1$ modulo $p$.
This follows from the fact that the exchange map $E(z_1,z_2)=(z_2,z_1)$ induces a homeomorphism between $L(p,q)$ and $L(p,\ell)$ when $q\ell \equiv 1$, and the composition $E\circ F$, where $F(z_1,z_2)=(z_1,\overline{z_2})$, does the same when $q\ell \equiv -1$. In either case, $K(p,q)$ and $K(p,\ell)$ are isotopic, and thus their knot groups are isomorphic.

The knot $K(p,q)$ is doubly amphichiral, meaning that there are two commuting involutions $h_1$ and $h_2$ induced by the transformations $(z_1,z_2)\mapsto(z_1,-z_2)$ and $(z_1,z_2)\mapsto (-z_1,z_2)$, respectively. Each $h_i$ fixes a circle that intersects the knot in two points, and the two circles meet at two points. Choosing a base point of $G$ at one of the intersection points of the two circles, the involutions $h_1$ and $h_2$ induce automorphisms $h_1^\ast$ and $h_2^\ast$ of $G$, respectively, which commute with each other.
Explicitly, they are given by
\begin{equation}
\label{eqn.h12}
    \left\{
    \begin{array}{l}
         h_1^\ast(u) = u^{-1} \\[2pt]
         h_1^\ast(v) = v^{-1}
    \end{array}
    \right. , \quad 
    \left\{
    \begin{array}{l}
         h_2^\ast(u) = v^{-1} \\[2pt]
         h_2^\ast(v) = u^{-1}
    \end{array}
    \right. \, .
\end{equation}
It is not trivial from the presentation~\eqref{presentation} that $h^\ast_1$ and $h^\ast_2$ are well-defined, one of the key properties of the sequence $\epsilon_k$ is that it is palindromic in the sense that $\epsilon_k=\epsilon_{p-k}$ for all $0<k<p$. 

\subsection{Parabolic character varieties}
\label{sec.para}

In this section, we consider the character variety of irreducible representations 
\begin{equation*}
    \rho:G = \pi_1(S^3\setminus K(p,q))\to \SL_2(\C)
\end{equation*}
which are \emph{parabolic}, by which we mean that each meridian is mapped to a non-trivial matrix with trace 2. It is well-known that this is an affine algebraic set $\calX$ defined over $\Z$ whose complex points can up to conjugation be put in the form 
$$\rho(u)=\begin{pmatrix} 1 & 1 \\ 0 & 1\end{pmatrix},\quad \rho(v)=\begin{pmatrix} 1 & 0 \\ X & 1\end{pmatrix}.$$
Note that $X$ is nonzero, as $u$ and $v$ are conjugate in $G$. The coordinate ring $\Q[\calX]$ of $\calX$ is generated by $\Tr \rho(u)$, $\Tr \rho(v)$, and $\Tr \rho(uv)$; hence, it has the single generator $X$. More precisely, it is shown in \cite{Riley} that
$\Q[\calX]$ is isomorphic to $\Q[X]/ (R)$
where $R \in \Z[X]$ is called the Riley polynomial of the knot $K(p,q)$.

For the purpose of this article, it is more elegant to adopt a symmetric setting by instead taking
\begin{equation}\label{rho}
\rho(u)=\begin{pmatrix} 1 & X \\ 0 & 1\end{pmatrix},\quad \rho(v)=\begin{pmatrix} 1 & 0 \\ X & 1\end{pmatrix}.
\end{equation}
Note that $X$ in this setting is also non-zero, as meridians are sent to non-trivial matrices.
This symmetric setting can be understood conceptually in the following way. 
Recall from Section~\ref{sec.top} that $S^3 \setminus K(p,q)$ admits two involutions $h_1$ and $h_2$. The quotient of $S^3\setminus K(p,q)$ by these involutions is an orbifold $\calO(p,q)$, where its (orbifold) fundamental group fits into the exact sequence 
\begin{equation*}
0 \to G \to \pi_1(\calO(p,q))\to (\Z/2\Z)^2\to 0 \,.
\end{equation*}
The above sequence splits. In addition, it follows from Equation~\eqref{eqn.h12} that 
\begin{equation}\label{conjugation}
    \left\{
    \begin{array}{l}
         h_1 u h_1^{-1} = u^{-1} \\[2pt]
         h_1 v h_1^{-1} = v^{-1}
    \end{array}
    \right. , \quad 
    \left\{
    \begin{array}{l}
         h_2 u h_2^{-1} = v^{-1} \\[2pt]
         h_2 v h_2^{-1} = u^{-1}
    \end{array}
    \right.  .
\end{equation}
We then add an extra global $\Z/2\Z$ isotropy group to $\pi_1(\calO(p,q))$ by considering the fiber product 
\begin{equation*}
    \hat{G}=\pi_1(\calO(p,q)) \! \underset{(\Z/2\Z)^2}{\times} \! Q_8 \,.
\end{equation*}
Here $Q_8$ is the standard quaternion group  and $Q_8\to (\Z/2\Z)^2$ is the abelianization map.
The group $\hat{G}$ is a central extension of $\pi_1(\calO(p,q))$, and the involutions $h_1$ and $h_2$ lift to elements $\hat{h}_1$ and $\hat{h}_2$ in $Q_8$, respectively, whose squares are equal to the central element of order 2.

\begin{Definition}
The \emph{parabolic character variety} $\hat{\calX}$ of $\hat{G}$ is the algebraic set of conjugacy classes of representations $\rho:\hat{G}\to \SL_2(\C)$ such that the its restriction  to $G$ is irreducible and parabolic, and $\rho(\hat{h}_1^2)=\rho(\hat{h}_2^2)=-I$.
\end{Definition}

For any $\rho$ in $\hat{\calX}$, the last condition implies that its restriction to $Q_8$ is irreducible, because any reducible representation of $Q_8$ satisfies $\rho(\hat{h}_i^2)=I$. Hence, there is a unique conjugation so that $Q_8$ has its usual irreducible representation in $\SL_2(\C)$:
\begin{equation}
    \label{eqn.standard}
    \rho(\hat{h}_1)=\begin{pmatrix} i & 0 \\ 0 &-i\end{pmatrix},\quad \rho(\hat h_2)=\begin{pmatrix} 0 & 1 \\ -1 &0\end{pmatrix},\quad \rho(\hat h_1 \hat h_2)=\begin{pmatrix} 0 & i \\ i & 0 \end{pmatrix} . 
\end{equation}

\begin{Proposition}
The conjugations described in Equation \eqref{conjugation}  force $\rho(u)$ and $\rho(v)$ to have the symmetric form given in \eqref{rho} up to exchanging the role of $u$ and $v$.
\end{Proposition}
\begin{proof}
    From the fact $\hat{h}_1 u \hat{h}_1^{-1} = u^{-1}$, we deduce that $\rho(u)_{11}=\rho(u)_{22}$. Here $A_{ij}$ denotes the $(i,j)$-entry of $A$. Combined with the condition $\Tr \rho(u)=2$, this implies $\rho(u)_{11} = \rho(u)_{22} =1$. Since $\det \rho(u)=1$, it follows that either $\rho(u)_{12}$ or $\rho(u)_{21}$ is zero. The same reasoning applies to $v$, showing that $\rho(v)$ is upper or lower triangular with ones on the diagonal. 
    Since conjugating with $h_2$ changes $u$ to $v^{-1}$, $\rho(u)$ and $\rho(v)$ can not be upper or lower triangular simultaneously, one must be upper and the other lower. Finally, the fact that the conjugation by $\hat{h}_1 \hat{h}_2$ exchanges $u$ and $v$ forces the non-zero off-diagonal entries of $\rho(u)$ and $\rho(v)$ to be equal. This results in Equation~\eqref{rho} up to exchanging the role of $u$ and $v$.
\end{proof}
 
As with any character variety, the coordinate ring $\Q[\hat{\calX}]$ of $\hat{\calX}$ is generated by trace functions $\Tr \rho(\gamma)$ for $\gamma\in \hat{G}$. As we have explicit generators of $\hat{G}$, it is easy to check that $\Q[\hat{\calX}]$ is generated by $x$ and $i$; for instance, we have
$\Tr \rho(\hat h_2 u) = -x$ and $\Tr \rho(u \hat h_1\hat h_2) = ix$.
In Proposition~\ref{prop.iota} below, we will prove that $i$ can be expressed as a polynomial in $x$, implying that $\Q[\hat{\calX}]$ has the single generator $x$.

Given formal variables $X_1,X_2,\ldots,X_k$, we define the continuants
\begin{equation*}
K(X_1,\ldots,X_k)=
\det
\begin{pmatrix}
X_1 & -1 &  & & \\
1   & X_2&-1& & \\
    &  \ddots  & \ddots &\ddots &  \\  
    &   & 1&X_{k-1}&-1 \\
    &    &  &1&X_k
\end{pmatrix}  \in \Z[X_1,\ldots,X_k]\,.
\end{equation*}
The following are well-known properties of continuants. We omit the proofs, as they follow readily by routine induction arguments.
\begin{Lemma} 
\label{lem.K}
The continuants $K(X_1,\ldots,X_k)$ satisfy
\begin{align}
     \label{eqn.recursion} K(X_1,\ldots,X_k)&=X_k K(X_1,\ldots,X_{k-1})+K(X_1,\ldots,X_{k-2}) ,\\[2pt]
     \nonumber K(X_1,\ldots,X_k)&=K(X_k,\ldots,X_1) , \\
     \label{eqn.matK} \begin{pmatrix} X_1 & 1 \\ 1 & 0\end{pmatrix}\cdots \begin{pmatrix} X_k & 1 \\ 1 & 0\end{pmatrix} &= \begin{pmatrix} K(X_1,\ldots,X_k) & K(X_1,\ldots,X_{k-1}) \\ K(X_2,\ldots,X_{k}) & K(X_2,\ldots,X_{k-1})\end{pmatrix}
\end{align}
for any $k \geq 1$. In addition, 
\begin{align*}
&(-1)^k K(X_1,\ldots, X_{i-1})\, K(X_{i+k+2},\ldots, X_{i+j})  \\
    & = K(X_1,\ldots,X_{i+j})K(X_{i+1},\ldots,X_{i+k})-K(X_1,\ldots,X_{i+k})K(X_{i+1},\ldots,X_{i+j})
\end{align*}
holds for any $i\ge 1$, $k\ge0$, and  $j\ge k + 1$.
\end{Lemma}

We set 
\begin{equation}
    \label{eqn.PQ}
    P_k=K(\epsilon_1 X,\ldots,\epsilon_{k}X), \quad Q_k=K(\epsilon_2X,\ldots,\epsilon_k X)\in \Z[X] 
\end{equation}
for $k\ge 1$ and extend both to all $k\in \Z$ by using the recursion~\eqref{eqn.recursion}; for instance, 
$$ P_0=1,\quad P_{-1}=0,\quad Q_0=0,\quad Q_1=1.$$
Equation~\eqref{eqn.matK} says that for any $k\ge 1$
\begin{equation}
\label{eqn.sym}
    \begin{pmatrix} P_k & P_{k-1}\\ Q_k & Q_{k-1}\end{pmatrix}=\begin{pmatrix} \epsilon_1 X & 1 \\ 1 & 0 \end{pmatrix} \cdots \begin{pmatrix} \epsilon_k X & 1 \\ 1 & 0 \end{pmatrix}\,.
\end{equation}
Grouping matrices in the right-hand side in pairs, we have 
\begin{equation*}
    \begin{pmatrix} \epsilon_k X & 1 \\ 1 & 0 \end{pmatrix} \begin{pmatrix} \epsilon_{k+1} X & 1 \\ 1 & 0 \end{pmatrix}=
    \begin{pmatrix} 1 & \epsilon_k X  \\ 0 & 1\end{pmatrix} \begin{pmatrix} 1 & 0 \\ \epsilon_{k+1} X & 1 \end{pmatrix}=\rho(u^{\epsilon_k}v^{\epsilon_{k+1}})
\end{equation*}
for any $k \geq 1$, implying that
\begin{align}
\label{eqn.PQmatrix}
    \begin{pmatrix} P_{2k} & P_{2k-1}\\ Q_{2k} & Q_{2k-1}\end{pmatrix}&=\rho(u^{\epsilon_1}v^{\epsilon_2}\cdots u^{\epsilon_{2k-1}}v^{\epsilon_{2k}})
\end{align}
as well as 
\begin{align}
\label{eqn.PQmatrix2}
    \begin{pmatrix} P_{2k} & P_{2k+1}\\ Q_{2k} & Q_{2k+1}\end{pmatrix}&=\rho(u^{\epsilon_1}v^{\epsilon_2}\cdots u^{\epsilon_{2k-1}} v^{\epsilon_{2k}}u^{\epsilon_{2k+1}}).
\end{align}
In particular, for $w=u^{\epsilon_1}v^{\epsilon_2}\cdots u^{\epsilon_{p-2}} v^{\epsilon_{p-1}}$, one has
\begin{equation}
\label{eqn.w}
    \rho(w)=\begin{pmatrix} P_{p-1} & P_{p-2}\\ Q_{p-1} & Q_{p-2}\end{pmatrix}.
\end{equation}
It is shown in \cite{Riley} that the Riley polynomial is equal to $\rho(w)_{11} = P_{p-1}$. That is, the assignment~
\eqref{rho} defines an $\SL_2(\C)$-representation of $G$ if and only if $P_{p-1}(X)=0$. 

\begin{Proposition}
\label{prop.iota}
$P_{p-2}$ satisfies $P_{p-2}^2 = -1$ in $\Q[X]/(P_{p-1})$.
\end{Proposition}
\begin{proof}
Combining Equation~\eqref{eqn.sym} for $k=p-1$ with the fact that $\epsilon_k = \epsilon_{p-k}$ for all $0<k<p$, we deduce that $\rho(w)$ is symmetric, i.e., $P_{p-2}=Q_{p-1}$. Together with the fact that $\rho(w)$ has determinant 1, we have $P_{p-2}^2=-1$ in $\Q[X]/(P_{p-1})$. 
\end{proof}

The above proposition shows that the complex number $i$ in~\eqref{eqn.standard} can be replaced by the polynomial $P_{p-2}\in \Q[X]/(P_{p-1})$. It follows that the coordinate ring $\Q[\hat{\calX}]$ of $\hat{\calX}$ has the single generator $X$.
\begin{Proposition}
The coordinate ring $\Q[\hat{\calX}]$ is isomorphic to $\Q[X]/(P_{p-1})$.
\end{Proposition}
\begin{proof}
    The coordinate ring $\Q[\calX]$ of $\calX$ is generated by $\Tr \rho(u)$, $\Tr \rho(v)$, and $\Tr \rho(uv)$ and is thus generated by $X^2$, subject to the relation given by the Riley polynomial. Namely, $\Q[\calX] \simeq \Q[X^2]/(P_{p-1})$. On the other hand, $\Q[\hat{\calX}]$ is obtained from $\Q[\calX]$ by adding $X$ and $i \in \C$, where the latter can be replaced by the polynomial $P_{p-2}$ due to Proposition~\ref{prop.iota}. Hence $\Q[\hat{\calX}]$ is isomorphic to $\Q[X]/ (P_{p-1})$.
\end{proof}



As Riley did in \cite{Riley} and the first author repeated in \cite{S2P}, by reducing modulo 2, one can prove that the polynomial $P_{p-1}$ has no multiple roots. In other words, $\Q[X]/(P_{p-1})$ is a semi-simple $\Q$-algebra.

\subsection{Computation of the torsion}

In this section, we compute the Reidemeister torsions of $K(p,q)$ associated with representations $\rho_n$ defined as below. The computation proceeds in $V=\Q[\hat{\calX}]=\Q[X]/(P_{p-1})$, and we denote by $x$, $\calP_k$ and $\calQ_k$ the image of $X$, $P_k$ and $Q_k$ in $V$, respectively.

Recall from Section~\ref{sec.para} that the polynomial $P_{p-1}$ is designed to give a parabolic representation $\rho : G \rightarrow \SL_2(V)$.
The natural action of $\SL_2(V)$ on the $n$-th symmetric power $\S^nV^2$ of $V^2$ defines a representation $s_n:\SL_2(V)\to \SL(\S^nV^2)$. The composition $\rho_n = s_n \circ \rho$ is a representation of $G$ into $\SL(\S^nV^2)$ and allows us to twist the usual cohomology of the knot complement  by the local system $\S^n V^2$. Under some cohomological conditions specified below, the Reidemeister torsion with respect to $\rho_n$ is well-defined (see e.g. \cite{Porti}), and we simply denote it by
\begin{equation*}
    \tau_n \in V / \{\pm1\} \,. 
\end{equation*}
Note that the torsion is defined up to sign. This sign ambiguity can be resolved by fixing a homological order of the knot complement, but we will not address it in this paper.

The fact that the torsion is invariant under simple homotopy allows us to compute the torsion $\tau_n$ using the 2-dimensional complex associated with the presentation~\eqref{presentation} of the knot group. Since this presentation has two generators $u$ and $v$ with a single relator $r=wuw^{-1}v^{-1}$, the corresponding chain complex is
\begin{equation}\label{cellcomplex}
    \S^nV^2 \overset{D^0}{\longrightarrow} \S^nV^2 \oplus \S^nV^2 \overset{D^1}{\longrightarrow} \S^nV^2 \, .
\end{equation}
The differentials $D^i$ of the above complex are given as follows.
Let $F_2$ be the free group with generators $u$ and $v$, and let $\frac{\partial}{\partial u},\frac{\partial}{\partial v}$ be the usual Fox derivatives, i.e, the unique endomorphisms of $\Z[F_2]$ such that
$$\frac{\partial u}{\partial u}=1,\quad \frac{\partial v}{\partial u}=0,\quad \frac{\partial u}{\partial v}=0,\quad \frac{\partial v}{\partial v}=1$$
and
$$\frac{\partial ab}{\partial u}=\frac{\partial a}{\partial u}+a\frac{\partial b}{\partial u},\quad
\frac{\partial ab}{\partial v}=\frac{\partial a}{\partial v}+a\frac{\partial b}{\partial v} \quad \text{for any } a,b\in F_2 \,.$$
Note that 
\begin{equation}
\label{eqn.funda}
     \frac{\partial a}{\partial u}(u-1)+\frac{\partial{a}}{\partial v}(v-1)=a-1 \quad \text{for any } a \in F_2 \,.
\end{equation}
Abusing notation, we denote by $\rho_n:\Z[F_2]\to \Z[G]\to \End(\S^n V^2)$ the map sending $a$ to $\rho_n(a)$. Then the differentials in the chain compelex~\eqref{cellcomplex} read
$$ D^0(\xi)=\Big(\rho_n(u-1)\xi,\ \rho_n(v-1)\xi \Big),\quad D^1(\xi,\eta)=\rho_n\! \left(\frac{\partial r}{\partial u}\right)\xi+\rho_n\! \left(\frac{\partial r}{\partial v}\right)\eta. $$
Note that $D^1\circ D^0=0$ follows from Equation~\eqref{eqn.funda} together with the fact that $\rho_n(r)=I$.

Let us further compute $D^1$.  
Recall that the word $w$ is given by $w=u^{\epsilon_1}\cdots v^{\epsilon_{p-1}}$. A striaghtforward computation shows that
\begin{align*}
    \frac{\partial w}{\partial u}&=  \sum_{k=1}^{\frac{p-1}{2}}
    \begin{cases}
        u^{\epsilon_1}v^{\epsilon_2}\cdots v^{\epsilon_{2k-2}} & \text{if }\epsilon_{2k-1}=1 \\
        -u^{\epsilon_1}v^{\epsilon_2}\cdots v^{\epsilon_{2k-2}}u^{\epsilon_{2k-1}} & \text{if }\epsilon_{2k-1}=-1 
    \end{cases}, \\[2pt]
    \frac{\partial w}{\partial v}&=\sum_{k=1}^{\frac{p-1}{2}}
    \begin{cases}
        u^{\epsilon_1}v^{\epsilon_2}\cdots u^{\epsilon_{2k-1}} &\text{if }\epsilon_{2k}=1 \\
        -u^{\epsilon_1}v^{\epsilon_2}\cdots u^{\epsilon_{2k-1}}v^{\epsilon_{2k}} & \text{if }\epsilon_{2k}=-1
    \end{cases}.
\end{align*}
Using Equations~\eqref{eqn.PQmatrix} and \eqref{eqn.PQmatrix2}, one has
\begin{align}
    \label{wu}
    \rho_n \! \left(\frac{\partial w}{\partial u}\right) &=\sum_{k=1}^{\frac{p-1}{2}}\epsilon_{2k-1} \, s_n\! \begin{pmatrix} \calP_{2k-2} & \calP_{2k-2-\epsilon_{2k-1}} \\ \calQ_{2k-2} & \calQ_{2k-2-\epsilon_{2k-1}}\end{pmatrix},\\
    \label{wv}
    \rho_n \! \left(\frac{\partial w}{\partial v}\right)&=\sum_{k=1}^{\frac{p-1}{2}}\epsilon_{2k} \, s_n \! \begin{pmatrix} \calP_{2k-1-\epsilon_{2k}} & \calP_{2k-1} \\ \calQ_{2k-1-\epsilon_{2k}} & \calQ_{2k-1}\end{pmatrix}.
\end{align}
Combining these equations with 
\begin{equation}\label{fox}
\frac{\partial r}{\partial u}= (1-v) \frac{\partial w}{\partial u}+w,\quad \frac{\partial r}{\partial v}=(1-v)\frac{\partial w}{\partial v}-1 \,,
\end{equation}
which directly follows from $r=wuw^{-1}v^{-1}$, we obtain an explicit formula for $D^1$.

\subsubsection{Tautological representation}

In this section, we compute the torsion $\tau_1$ associated with the tautological representation $\rho_1$. Note that $s_1$ is the identity map and thus $\rho_1=\rho$.

\begin{Lemma} 
    \label{lem.DD}
    Writing the differentials $D^i$ in matrix form,  we have
    $$D^0=\left(\begin{array}{cc}0 &x \\ 0& 0 \\ \hline 0 & 0 \\ x & 0 \end{array}\right),\quad D^1=\left(\begin{array}{cc|cc}0& D^{1}_{1} & D^1_{2} & 0\\ 0 & D^1_{3} & D^1_{4} & 0\end{array}\right)$$
    where  $D^1_{1}=\calP_{p-2}$, $D^1_{2}=-1$, 
    \begin{equation*}
        D^1_{3}=\calQ_{p-2}-x\sum_{k=1}^{\frac{p-1}{2}}\epsilon_{2k-1}\calP_{2k-2-\epsilon_{2k-1}} \,, \quad D^1_{4}=-x\sum_{k=1}^{\frac{p-1}{2}}\epsilon_{2k}\calP_{2k-1-\epsilon_{2k}} \, .
    \end{equation*}
\end{Lemma}
\begin{proof}
    It is clear from
    \begin{equation*}
        \rho_1(u-1) = \begin{pmatrix}
            0 & x \\ 0 & 0
        \end{pmatrix}, \quad \,
        \rho_1(v-1) = \begin{pmatrix}
            0 & 0 \\ x & 0
        \end{pmatrix}
    \end{equation*}
    that the matrix form of $D^0$ is given as in the claim.
    
    Equation~\eqref{eqn.funda} for $a=r$ says that  $\frac{\partial r}{\partial u}(u-1)+\frac{\partial r}{\partial v}(v-1)=0$. Thus we can write a priori
    $$\rho_1 \left(\frac{\partial r}{\partial u}\right)=\begin{pmatrix} 0 & a \\ 0 & b\end{pmatrix}, \quad \rho_1\left(\frac{\partial r}{\partial v}\right)=\begin{pmatrix} c & 0 \\ d & 0\end{pmatrix} \, .$$
    It then follows from Equations \eqref{fox} and \eqref{eqn.w} that $a=\calP_{p-2}=D^1_{1}$, $c=-1=D^1_{2}$, and from~\eqref{wu} and \eqref{wv} that $b=D^1_{3}$ and $d=D^1_{4}$.
\end{proof}

The above lemma implies that the chain complex~\eqref{cellcomplex} is acyclic if and only if $D^1_{1} D^1_{4}-D^1_{2} D^1_{3} \neq 0$. In the acyclic case, the torsion $\tau_1$ can be computed directly by using the Cayley formula (see, e.g. \cite[p.485]{GKZ}). Precisely, we have
\begin{equation}
\label{eqn.tau}
    \tau_1 = \frac{\Delta_1}{\Delta_0} \quad \text{where} \quad \Delta_0=x^2, \ \Delta_1=D^1_{1}D^1_{4} - D^1_{2} D^1_{3}.
\end{equation} 

\begin{Theorem}
\label{torfunda}
The torsion of the two bridge knot $K(p,q)$ in its tautological representation $\rho_1$ is well-defined and given by $$\tau_1= \frac{1}{x^2}\left(\calQ_{p-2}-x\sum_{k=1}^{\frac{p-1}{2}}\epsilon_{2k-1}(\calP_{2k-1}+\calP_{2k-3})\right).$$
\end{Theorem}
\begin{proof}
The last equality of Lemma~\ref{lem.K} implies that
$\calP_{p-2}\calP_k=(-1)^k\calP_{p-2-k}$ for any $k$.
Hence we can rewrite
$$\sum_{k=1}^{\frac{p-1}{2}} \epsilon_{2k} \calP_{p-2} \calP_{2k-1-\epsilon_{2k}}=\sum_{k=1}^{\frac{p-1}{2}}\epsilon_{2k}\calP_{p-1-2k+\epsilon_{2k}}=\sum_{k=0}^{\frac{p-3}{2}}\epsilon_{2k+1}\calP_{2k+\epsilon_{2k+1}} \, .$$
In the last equality, we have substituted $2k$ with $p-2k-1$ and used the symmetry $\epsilon_{2k+1}=\epsilon_{p-2k-1}$. Then two sums that appear in $$\Delta_1= -x\sum_{k=1}^{\frac{p-1}{2}}\epsilon_{2k} \calP_{p-2} \calP_{2k-1-\epsilon_{2k}}+\calQ_{p-2}-x\sum_{k=1}^{\frac{p-1}{2}}\epsilon_{2k-1}\calP_{2k-2-\epsilon_{2k-1}} $$ add to $\sum_k\epsilon_{2k-1}(\calP_{2k-1}+\calP_{2k-3})$. Combining this with Equation~\eqref{eqn.tau}, we obtain the desired formula of $\tau_1$.

It remains to show that $\Delta_1$ is invertible to prove that the chain complex~\eqref{cellcomplex} is acyclic and the torsion $\tau_1$ is well-defined. This follows from Theorem~\ref{thm.simp} below, saying that $\Delta_1=x^2\tau_1$ is invertible with the explicit inverse $\calP_{\ell-1}/2$. 
\end{proof}

\subsubsection{Simplication of $\tau_1$}

In this section, we show that the formula of $\tau_1$ in Theorem~\ref{torfunda} admits a remarkable simplification.
\begin{Theorem}
\label{thm.simp}
One has 
\begin{equation*}
    \tau_1=\frac{2}{x^2 \calP_{\ell-1}}
\end{equation*}
where $0<\ell<p$ is the unique odd integer satisfying $q\ell \equiv \pm 1$ modulo $p$.
\end{Theorem}

Prior to the proof, we record some lemmas. 
\begin{Lemma}
\label{lem.R}
Let $R_k$ be a sequence of polynomials satisfying the recursion
\begin{equation}\label{rec}
R_k=\epsilon_k x R_{k-1}+R_{k-2} .
\end{equation}
Then $R_k = R_0P_k + R_{-1}Q_k$ for any $k$.
\end{Lemma}
\begin{proof}
Recall that both $P_k$ and $Q_k$ satisfy the same recursion~\eqref{rec} with $P_{-1}=0$, $P_0=1$, and $Q_{-1}=1$, $Q_0=0$. It follows that $R_k$ and $R_0P_k + R_{-1}Q_k$ satisfy the same recursion with the same values for $k=0$ and $-1$. Therefore, they coincide for any $k$.
\end{proof}

We henceforth set $\iota = \calP_{p-2}$; by Proposition~\ref{prop.iota}, we have $\iota^{2} = -1$.

\begin{Lemma}
\label{lem.e1}
    One has $\calP_{k+p}=(-1)^{k}\iota \calP_k$ and $\calP_{k+2p} = \calP_k$ for any $k$.
\end{Lemma}
\begin{proof}
Since $\epsilon_{k+p} = -\epsilon_k$, it follows that the sequence $P_{k+p}(-x) = (-1)^{k+p} P_{k+p}(x)$ satisfies the recursion~\eqref{rec}. By Lemma~\ref{lem.R}, we have
$$(-1)^{k+p}P_{k+p}= -P_p P_{k}+P_{p-1}Q_k.$$
As $P_p=\epsilon_p x P_{p-1}+P_{p-2}$, we obtain the relation $\calP_{k+p}=(-1)^{k}\iota \calP_k$ in $V$. Applying it twice, one has $\calP_{k+2p}=\calP_k$.
\end{proof}

\begin{Lemma}
\label{lem.e2}
One has $\calP_{-2-k}=\calP_k$ for any $k$.
\end{Lemma}
\begin{proof}
Since $\epsilon_{-k}=-\epsilon_k$ for $p\nmid k$, the sequence $P_{-2-k}$ satisfies the recursion \eqref{rec} for $0<k<p$. By Lemma~\ref{lem.R}, $P_{-2-k}=P_{-2}P_k+P_{-1}Q_k = P_{k}$ holds for $0 < k <p$, which also holds for $k=0$. Using Lemma~\ref{lem.e1}, the formula extends to all $k$ in $V$.
\end{proof}

\begin{Lemma}\label{omega}
One has $\calP_{\ell'-1} \calQ_{p-2}=-\iota(\calP_{\ell'}+\calP_{\ell'-2})$
where $\ell'$ is the unique odd integer satisfying $0<\ell'<2p$ and $\ell'q \equiv-1$ modulo $2p$. 
\end{Lemma}
\begin{proof}
One easily proves that $\epsilon_{\ell'+k}=\epsilon_k$ if $p\nmid k$. This implies that $P_{\ell'+k}$ satisfies the recursion \eqref{rec} for $0<k<p$ and thus, by Lemma~\ref{lem.R}, 
\begin{equation*}
P_{\ell'+k}=P_{\ell'}P_k+P_{\ell'-1}Q_k.
\end{equation*}
for $0<k<p$.  
Taking $k=p-2$, we have $P_{\ell'+p-2}=P_{\ell'}P_{p-2}+P_{\ell'-1}Q_{p-2}$ which, combined with Lemma~\ref{lem.e1}, gives $-\iota \calP_{\ell'-2}=\iota\calP_{\ell'}+\calP_{\ell'-1} \calQ_{p-2}$.
\end{proof}

\begin{Lemma}\label{sumoftwo}
For all $-1\le k<p$ one has $$(P_{\ell'}-P_{\ell'-2}) P_k=P_{\ell'+k}-P_{\ell'-k-2}.$$ The same equation holds with a minus sign for $-p <k<0$. 
\end{Lemma}
\begin{proof}
We prove it by induction on $k$: it holds for $k=-1$ as $P_{-1}=0$ and obviously for $k=0$.
Take now $0<k<p$ and suppose the formula holds for $k-1$ and $k-2$. Then one has 
\begin{align*}
(P_{\ell'}-P_{\ell'-2})P_k&=\epsilon_kx(P_{\ell'}-P_{\ell'-2})P_{k-1}+(P_{\ell'}-P_{\ell'-2})P_{k-2}\\
&= \epsilon_k x(P_{\ell'+k-1}-P_{\ell'-k-1})+P_{\ell'+k-2}-P_{\ell'-k}\\
&= \epsilon_k(\epsilon_{\ell'+k}(P_{\ell'+k}-P_{\ell'+k-2})-\epsilon_{\ell'-k}(P_{\ell'-k}-P_{\ell'-k-2}))\\
&\quad +P_{\ell'+k-2}-P_{\ell'-k}\,.
\end{align*}
As $\epsilon_{\ell'+k}=\epsilon_k$ and $\epsilon_{\ell'-k}=-\epsilon_k$ for $0< k<p$, the last expression simplifies and gives the result.  The same induction argument is applied for $-p<k<0$.
\end{proof}

\begin{proof}[Proof of Theorem~\ref{thm.simp}]
First observe that if $0 < \ell' < p$, then $\ell = \ell'$; otherwise, $\ell = 2p - \ell'$. In either case, by Lemmas~\ref{lem.e1} and \ref{lem.e2}, we have $\calP_{\ell-1} = \calP_{\ell'-1}$.

Due to Theorem \ref{torfunda}, it suffices to show that 
\begin{equation*}
    \calP_{\ell'-1}(\calQ_{p-2}-xS)=2 \quad \text{where } S=\sum_{k=1}^{\frac{p-1}{2}}\epsilon_{2k-1}(\calP_{2k-1}+\calP_{2k-3}) .
\end{equation*}
Applying Lemma \ref{omega} with the formula $-x \calP_{\ell'-1}=\calP_{\ell'}-\calP_{\ell'-2}$, the above equality is equivalent to
\begin{equation}
\label{eqn.goal}
    (\calP_{\ell'}-\calP_{\ell'-2})S=2+\iota(\calP_{\ell'}+\calP_{\ell'-2}).
\end{equation}
On the other hand, using the symmetry $\calP_{2k-3}=\calP_{-2k+1}$ and $\epsilon_{2k-1}=-\epsilon_{-2k+1}$, we can write 
\begin{equation*}
S= \sum_{k=1}^{\frac{p-1}{2}}(\epsilon_{2k-1}\calP_{2k-1}-\epsilon_{-2k+1}\calP_{-2k+1}) 
= \sum_{-p <k<p} \delta_k \epsilon_k \calP_k 
\end{equation*}
where $\delta_k=1$ for $0<k<p$ odd, $\delta_k=-1$ for $-p<k<0$ odd, and $\delta_k=0$ otherwise. As Lemma \ref{sumoftwo} states that $(\calP_{\ell'}-\calP_{\ell'-2})\calP_k=\delta_k(\calP_{\ell'+k}-\calP_{\ell'-k-2})$ for odd $-p<k<p$, we have
\begin{equation*}
  (\calP_{\ell'}-\calP_{\ell'-2})S=\sum_{k:\text{odd}}(\epsilon'_k \calP_{\ell'+k}-\epsilon'_k \calP_{\ell'-k-2})=\sum_{k:\text{even}}(\epsilon'_{k-\ell'}-\epsilon'_{\ell'-2-k})\calP_k.  
\end{equation*}
In this formula, $\epsilon'_k = \epsilon_k$ if $p \nmid k$, and $\epsilon'_k = 0$ otherwise; the sums are taken over $k \in \mathbb{Z}/2p\mathbb{Z}$, which are well-defined by Lemma~\ref{lem.e1}. As $\calP_{2k}=\calP_{2p-2-2k}$ and $\calP_{p-1}=0$, we can pack these in the form 
\begin{equation}
\label{eqn.goal2}
    (\calP_{\ell'}-\calP_{\ell'-2})S=\sum_{k=0}^{\frac{p-3}{2}}(u_k +v_k)\calP_{2k}
\end{equation} 
where $u_k=\epsilon'_{2k-\ell'}-\epsilon'_{\ell'+2k}$ and $v_k=\epsilon'_{-2-2k-\ell'}-\epsilon'_{\ell'-2-2k}$.

\begin{itemize}[leftmargin=*]
    \item For $k=0$, one has $\epsilon'_{-\ell'} + \epsilon'_{-2-\ell'}=0$, $\epsilon'_{\ell'}=\epsilon'_{\ell'-2}=-1$, and $\calP_0=1$. It follows that  $(u_0+v_0)\calP_0=2$.
    \item For $k>0$ one has $\epsilon_{\ell'+2k}=\epsilon_{2k}=\epsilon_{2k-\ell'}$ hence $u_k=0$ unless $\ell'+2k \equiv p$ or $2k-\ell'\equiv p\pmod{2p}$: 
    \subitem If $\ell'<p$, then there is $k_0$ such that $\ell'+2k_0=p$, and one has $u_{k_0}\calP_{2k_0}=\epsilon_{2k_0-\ell'}\calP_{2k_0}=\epsilon_{p-2\ell'}\calP_{p-\ell'}=-\calP_{p-\ell'}=\iota \calP_{\ell'-2}$;
    \subitem If $\ell'>p$, then there is $k_0$ such that $2k_0-\ell'=-p$, and one has $u_{k_0} \calP_{2k_0}=-\epsilon_{\ell'+2k_0} \calP_{2k_0}=\epsilon_{\ell'-p}\calP_{\ell'-p}=-\calP_{\ell'+p}=\iota \calP_{\ell'}$.
    \item Similarly, one has $v_k=0$ for $k>0$ unless  $\ell'-2-2k \equiv p$ or $-2-2k -\ell'\equiv p \pmod{2p}$: 
    \subitem If $\ell'<p$, then there is $k_0$ such that $-2-2k_0-\ell'=-p$, and one has $v_{k_0} \calP_{2k_0}=-\epsilon_{\ell'-2-2k_0}\calP_{2k_0}=\iota \calP_{\ell'}$;
    \subitem  If $\ell'>p$, then there is $k_0$ such that $\ell'-2-2k_0=p$, and one has  $v_{k_0} \calP_{2k_0}=\epsilon_{-2-2k_0-\ell'}\calP_{2k_0}=\iota \calP_{\ell'-2}$. 
\end{itemize}
We conclude that in both cases $\ell'<p$ and $\ell'>p$, the sum in~\eqref{eqn.goal2} results in $2+\iota(\calP_{\ell'}+\calP_{\ell'-2})$. This proves Equation~\eqref{eqn.goal} and thus the theorem.
\end{proof}

\subsubsection{Adjoint representation}

In this section, we compute the torsion $\tau_2$ associated with the adjoint representation $\rho_2 = s_2 \circ \rho$.

The symmetric power $\S^2V^2$ is 3-dimensional and the representation $s_2$ is known to be isomorphic to the adjoint representation $\SL_2(V)\to \SL(\sl_2(V))$, mapping $g$ to the map $\xi\mapsto g\xi g^{-1}$. 
We denote by $\mathfrak{e}_1, \mathfrak{e}_2,\mathfrak{e}_3$ the standard basis of $\S^2V^2$ and by $\langle \cdot, \cdot \rangle : \S^2V^2 \otimes \S^2V^2 \rightarrow V$ the canonical invariant pairing on $\S^2V^2$.

The torsion $\tau_2$ is our main case of interest. However, it provides extra difficulties due to the fact that the chain complex~\eqref{cellcomplex} is not acyclic. That is, $\mathrm{H}^\ast(M,\S^2 V^2)$ does not vanish where $M$ is the knot complement of $K(p,q)$ and the coefficient $\S^2V^2$ is twisted by $\rho_2 = s_2 \circ \rho$.
More precisely, it is well-known that $\mathrm{H}^k(M,\S^2V^2)$ vanishes except for $k=1,2$. In this non-acyclic case, defining the torsion requires fixing a basis of the cohomology \cite{PortiThesis}.

Let $\partial M$ be the boundary torus of $M$ and let $m\subset \partial M$ be a meridian of $K(p,q)$ homotopic to $u\in G$. The inclusion maps induce the following homomorphisms
\begin{equation}
\label{eqn.r}
  \mathrm{H}^1(M,\S^2V^2) \to \mathrm{H}^1(m,\S^2V^2), \quad
  \mathrm{H}^2(M,\S^2V^2)\to \mathrm{H}^2(\partial M,\S^2V^2) \,.
\end{equation}
It is well-known that when $K(p,q)$ is hyperbolic and $\rho$ is a lift of the geometric representation, these two maps are isomorphism. In our case, this will follow from a computation. 

Applying the Poincar\'{e} duality, combined with the invariant pairing of $\S^2V^2$, to the isomorphisms in  \eqref{eqn.r}, we obtain
\begin{equation}
    \label{eqn.isom}
    \mathrm{H}^1(M,\S^2V^2)\simeq \mathrm{H}^0(m,\S^2V^2)^*, \quad \mathrm{H}^2(M,\S^2V^2)\simeq \mathrm{H}^0(\partial M,\S^2V^2)^*\,.
\end{equation}
By definition of the 0-th cohomology, $\mathrm{H}^0(m,\S^2V^2)$ and $\mathrm{H}^0(\partial M,\S^2V^2)$ are generated by elements of $\S^2V^2$ that are fixed by the restrictions $\rho|_m$ and $\rho|_{\partial M}$, respectively. It follows that both are 1-dimensional and may be assume to be generated by $\mathfrak{e}_1$. We specify $\mathfrak{e}_1$ as a basis for both $\mathrm{H}^0(m,\S^2V^2)$ and $\mathrm{H}^0(\partial M,\S^2V^2$). This determines a basis of $\mathrm{H}^\ast(M;\S^2V^2)$ through the isomorphisms in~\eqref{eqn.isom}. With this choice of basis, the torsion $\tau_2$ is well-defined. 

\begin{Theorem}\label{toradj}
    The torsion of the two-bridge knot $K(p,q)$ in the adjoint representation $\rho_2$ and with the homological marking referred to above is given by 
    $$\tau_2=\frac{1}{2x^2\calP_{\ell-1}^2}\sum_{k=1}^{\frac{p-1}{2}}\epsilon_{2k}\calP_{2k-1}^2$$
    where $0<\ell<p$ is the unique odd integer satisfying $q\ell \equiv \pm 1$ modulo $p$.
\end{Theorem}

\begin{proof}
    Let $C^\ast$ be the chain complex~\eqref{cellcomplex} and $\mathrm{H}^\ast$ be its homology. We first define a quasi-isomorphism $f_\ast:C^*\to \mathrm{H}^\ast$.
    \begin{equation*}
        \begin{tikzcd}
        C^\ast : \ar["f_\ast",d] & \S^2V^2 \ar[r,"D^1"] \ar[d]  & \S^2V^2\oplus \S^2V^2 \ar[r,"D^2"]\ar[d,"f_1"]    & \S^2V^2 \ar[d,"f_2"]\\ 
        \mathrm{H}^\ast : & 0 \ar[r, "0"]   & V \ar[r,"0"]          & V
        \end{tikzcd}    
    \end{equation*}
    Note that $\S^2 V^2$ is isomorphic to $V^3$ with basis $\mathfrak{e}_1,\mathfrak{e}_2,\mathfrak{e}_3$.
    
    Since $m$ is a 1-subcomplex of the 2-complex used to define $C^\ast$, one can choose $f_1$ by mapping $(\xi,\eta)$ to $\langle \xi, \mathfrak{e}_1 \rangle$. In the standard basis, we have $f_1=(0,0,1,0,0,0)$. 

    For the second one, we consider $\partial M\simeq S^1\times S^1$ and a cellular map homotopic to the inclusion $\partial M \hookrightarrow M$. The boundary torus has a cell decomposition corresponding to the presentation  $$\pi_1(\partial M) \simeq \Z^2=\langle a,b\,|\,R\rangle, \quad  R=aba^{-1}b^{-1}.$$
    The natural morphism $\pi_1(\partial M) \to  \pi_1(M)$ sends $a$ to the meridian $u$ and $b$ to the (blackboard-framed) longitude $w^*w$. Here $w^*$ is the word $w$ read from right to left; see \cite[Prop.1] {Riley}. Note that this morphism is well-defined due to the following equality that holds in $F_2$: 
    \begin{equation}
        \label{eqn.g}
        [u,w^*w]=[g,r][r,w^*], \quad g=u^{\epsilon_0}v^{\epsilon_1}\cdots u^{\epsilon_{\ell'-1}}v^{\epsilon_{\ell'}}
    \end{equation} where $0<\ell'<2p$ is the inverse of $-q$ modulo $2p$; see Lemma 2 in \cite{vol2ponts}. This allows us to write down explicitly the morphism of complex inducing the restriction map $\mathrm{H}^*(M,\S^2V^2)\to \mathrm{H}^*(\partial M,\S^2 V^2)$ as follows:
    \begin{equation*}
        \begin{tikzcd}[column sep=3cm]
        \S^2V^2 \ar[r,"{(\rho_2(u-1),\rho_2(v-1))}"] \ar["=",d]  & \S^2V^2\oplus \S^2V^2 \ar[r,"{(\rho_2(\frac{\partial r}{\partial u}),\rho_2(\frac{\partial r}{\partial v}))}"]\ar[d,"\Phi"]    & \S^2V^2 \ar[d,"\Psi"]\\
        \S^2V^2  \ar[r, "{(\rho_2(a-1),\rho_2(b-1))}"]   & \S^2V^2 \oplus \S^2V^2 \ar[r,"{(\rho_2(\frac{\partial R}{\partial a}),\rho_2(\frac{\partial R}{\partial b}))}"]          & \S^2V^2
        \end{tikzcd}    
    \end{equation*}
    In this diagram, $\Phi$ is the Jacobian matrix of the transformation $(a,b)\mapsto(u,w^*w)$ and,  $\Psi$ sends $\xi$ to $\rho_2(g)\xi-\rho_2(w^*)\xi$; see Equation~\eqref{eqn.g}. 
    Recall that we chose the basis of $\mathrm{H}^0(\partial M,\S^2V^2)$ as $\mathfrak{e}_1$, hence the map $f_2 : \S^2V^2 \rightarrow V$ is given by
    $$f_2(\xi) = \langle \rho_2(g)\xi \, ,\, \mathfrak{e}_1 \rangle - \langle \rho_2(w^\ast)\xi \, ,\, \mathfrak{e}_1 \rangle = \langle \xi \, ,\, \rho_2(g)^{-1}\mathfrak{e}_1 \rangle - \langle \xi \, ,\, \rho_2(w^\ast)^{-1}\mathfrak{e}_1 \rangle .$$ 
    Recall from Equation~\eqref{eqn.w} that $\rho(w)$ has entries $\calP_{p-1}=0$, $\calP_{p-2}=\calQ_{p-1}$, and $\calQ_{p-2}$ where $\calP_{p-2}^2=-1$. From the symmetry between $w$ and $w^\ast$, we have
    \begin{equation}
    \label{eqn.wstar}
        \rho(w^\ast) = \begin{pmatrix} \calQ_{p-2} & \calP_{p-2}\\ \calP_{p-2} & 0\end{pmatrix}.
    \end{equation}
It follows that $\rho_2(w^\ast)^{-1}\mathfrak{e}_1= - \mathfrak{e}_3$.
On the other hand, we have 
$$\rho(g)=\begin{pmatrix} K(\epsilon_0 x,\ldots,\epsilon_{\ell'} x) & K(\epsilon_0 x,\ldots,\epsilon_{\ell'-1}x) \\ K(\epsilon_1 x,\ldots,\epsilon_{\ell'} x) & K(\epsilon_1 x,\ldots,\epsilon_{\ell'-1} x)\end{pmatrix} = \begin{pmatrix}
    * & * \\
    \calP_{\ell'} & \calP_{\ell'-1}
\end{pmatrix}.$$
It follows that $\rho_2(g)^{-1} \mathfrak{e}_1 = \calP_{\ell'-1}^2 \mathfrak{e}_1 + 2 \calP_{\ell'-1} \calP_{\ell'} \mathfrak{e}_2 - \calP_{\ell'}^2 \mathfrak{e}_3$. Therefore, 
$$
f_2(\xi) = \langle \xi \, , \, \calP_{{\ell'}-1}^2 \mathfrak{e}_1 + 2\calP_{{\ell'}-1}\calP_{\ell'} \mathfrak{e}_2 - (\calP_{\ell'}^2-1) \mathfrak{e}_3 \rangle ,
$$
showing that $f_2 = (\ast, \, \ast, \, \calP_{\ell'-1}^2)$ in the matrix form.
On the other hand, if $\ell \equiv -q^{-1} \pmod{p}$, then $\ell'=\ell$; otherwise, $\ell' = 2p - \ell$. In the latter case, the identity $\calP_{2p-k} = \calP_{k-2}$ implies that $\calP_{\ell' - 1} = \calP_{2p - \ell - 1}=\calP_{\ell-1}$. Thus, in either cases, we have $\calP_{\ell'-1} = \calP_{\ell-1}$ and thus $f_2 = (\ast, \, \ast, \, \calP_{\ell-1}^2)$

Finally, we compute the torsion $\tau_2$ from the acyclic chain complex associated with the mapping cone of $f_\ast : C^\ast \rightarrow \mathrm{H}^\ast$:
\begin{equation*}
    \begin{tikzcd}[column sep=1.5cm]
        \S^2 V^2  \ar[r,"{(u-1,v-1)}","D_0"'] & \S^2V^2 \oplus \S^2V^2 \ar[r,"{(\frac{\partial r}{\partial u}, \frac{\partial r}{\partial v},f_1)}","D_1"'] & \S^2V^2 \oplus V \ar[r,"{(f_2,0)}","D_2"'] & V \,.
    \end{tikzcd}
\end{equation*}
As computed, the differentials $D^i$ have matrix form
$$D^0=\left(\begin{array}{ccc}0 &-x &-x^2\\ 0& 0 &2x\\ 0 & 0 & 0 \\ \hline 0 & 0 &0\\ -2x & 0 & 0 \\ -x^2 & x & 0\end{array}\right),\quad D^1=\left(\begin{array}{ccc|ccc}&&&&&\\&\rho_2(\frac{\partial r}{\partial u})&&& \rho_2(\frac{\partial r}{\partial v})&\\&&&&& \\\hline 0&0&1&0&0&0 \end{array}\right)$$
and $D^2=\left(*\,*\,P_{\ell-1}^2\, 0\right)$.
To apply the Cayley formula, we extract the rows $1,2,5$ of $D^0$ to get a determinant $\Delta_0$, we extract the (complementary) columns $3,4,6$ and the rows $1,2,6$ of $D_1$ to get a determinant $\Delta_1$ and keep the third entry of $D_2$ to get $\Delta_2$. Then the Cayley formula (see, e.g. \cite[p.485]{GKZ}) says that
\begin{equation}
\label{eqn.delta}
\tau_2 = \frac{\Delta_1}{\Delta_0 \Delta_2}
\end{equation}
where $\Delta_0 = 4x^3$, $\Delta_1 = -2x(\sum_{k=1}^{\frac{p-1}{2}} \epsilon_{2k}\calP_{2k-1}^2) $, and $\Delta_2 = \calP_{\ell-1}^2$. Here $\Delta_1$ is computed from (see Equation~\eqref{fox})
$$
\rho_2 \left( \frac{\partial r}{\partial v} \right) = \begin{pmatrix}
    0 & 0 & 0 \\
    2x & 0 & 0 \\
    x^2 & -x & 0
\end{pmatrix}
\begin{pmatrix}
    \ast & \ast & a \\
    \ast & \ast & \ast \\
    \ast & \ast & \ast
\end{pmatrix} - I = \begin{pmatrix}
    -1 & 0 & 0 \\
    \ast & \ast & -2xa \\
    \ast & \ast & \ast
\end{pmatrix}
$$
where $a=-\sum_{k=1}^{\frac{p-1}{2}} \epsilon_{2k}\calP_{2k-1}^2$, which follows from Equation~\eqref{wv}. Substituting $\Delta_i$  into the Cayley formula~\eqref{eqn.delta}, we obtain the desired formula.

To show acyclicity of the cone complex, it is sufficicient to show that no $\Delta_i$ vanish. The non trivial one is $\Delta_1$, and its relation with the invertible element $\Omega$ in Section~\ref{sec.frob} will show that it is invertible; see Equation~\eqref{toromega}. 
\end{proof}

\subsection{Comparing reciprocal pairs}
\label{sec.dual}

Let $0<\ell<p$ be the unique odd integer such that $\ell q\equiv\pm 1 \pmod{p}$. Recall from Section~\ref{sec.top} that  the knots $K(p,q)$ and $K(p,\ell)$ are isotopic. It follows that the orbifolds $\calO(p,q)$ and $\calO(p,\ell)$ are homeomorphic, and the algebras $V$ and $V^\ast$ associated to $(p,q)$ and $(p,\ell)$, respectively, are isomorphic.
It is useful to make this algebra isomorphism explicit for two reasons. First, it clarifies the denominator of the formula in Theorems~\ref{thm.simp} and \ref{toradj}, providing a topological explanation for the discrepancy between the torsions and the element $\Omega$ in Section~\ref{sec.frob} below that will be needed to compute signatures. Second, it provides one of the few ways to verify the formulas in these sections: since  torsions are topological invariants, they should correspond under the isomorphism. 

\subsubsection{Explicit isomorphism}

\begin{Proposition}    
\label{thm.dual}
Let $\ell$ be the integer as above and let $V=\Q[X]/(P_{p-1})$ and $V^\ast=\Q[X^\ast]/(P^\ast_{p-1})$ the algebras associated to $K(p,q)$ and $K(p,\ell)$, respectively. Then the maps  
$$V^\ast\to V, \ x^\ast\mapsto x \calP_{\ell-1}(x) \quad \text{and} \quad V\to V^\ast, \  x\mapsto x^\ast \calP^\ast_{q-1}(x^\ast)$$
are algebra isomorphisms, inverse to each other.
\end{Proposition}
\begin{proof}
    Recall from Section~\ref{sec.top} that we can decompose the 3-sphere into two ball $B_1$ and $B_2$ where the knot $K(p,q)$ becomes two ramified arcs in each ball. Let us draw Schubert's normal form of $K(p,q)$ in the $\R^2$-plane so that its vertical edges run from $(0,0)$ to $(0,1)$ and from $(1,0)$ to $(1,1)$, respectively, as in Figure~\ref{fig.schubert_proof}. From Equation~\eqref{rho}, we have $\Tr\rho(uv)=2+x^2$ where $uv$ is a loop in the ball $B_1$ that encloses the two ramified arcs (the vertical edges of the normal form). 
    \begin{figure}[htpb!]
        \centering
        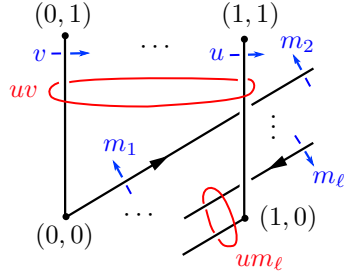
        \caption{Wirtinger generators in the normal form.}
        \label{fig.schubert_proof}
    \end{figure}

    In the other ball $B_2$, the two ramified arcs appear as arcs of slope $q/p$.  If we follow the diagram from $(0,0)$ to $(1,1/p)$, i.e., to the crossing just above $(1,0)$, the associated Wirtinger generator begins as  $v$ and is conjugated whenever it crosses a vertical edge. Let us denote this sequence of Wirtinger generators by $m_1,m_2,\ldots$ as in Figure~\ref{fig.schubert_proof}. Then we have
    $$ 
       m_k = (u^{\epsilon_1}v^{\epsilon_2} \cdots v^{\epsilon_{k-1}})^{-1}\, v \,(u^{\epsilon_1}v^{\epsilon_2} \cdots v^{\epsilon_{k-1}}).
    $$
    When we arrive to the crossing at $(1,1/p)$, the corresponding Wirtinger generator is $m_\ell$ 
    and Equation~\eqref{eqn.PQmatrix} yields 
    \begin{align*}
        \rho(m_\ell) & = \begin{pmatrix}
        \calQ_{\ell-2} & - \calP_{\ell-2} \\
        -\calQ_{\ell-1} & \calP_{\ell-1} 
    \end{pmatrix}
    \begin{pmatrix}
        1 & 0 \\
        x & 1
    \end{pmatrix}
    \begin{pmatrix}
        \calP_{\ell-1} & \calP_{\ell-2} \\
        \calQ_{\ell-1} & \calQ_{\ell-2} 
    \end{pmatrix}\\
    &  =
    \begin{pmatrix}
        1-x \calP_{\ell-2}\calP_{\ell-1} & -x \calP^2_{\ell-2} \\
        x \calP_{\ell-1}^2 & 1+x  \calP_{\ell-2}\calP_{\ell-1}
    \end{pmatrix}.
    \end{align*}
    Since the vertical edge from $(1,0)$ to $(1,1)$ is attached with the arc of slope $q/p$ at $(1,0)$, the loop in $B_2$ enclosing the two ramified arcs is given by $u m_{\ell}$ as depicted in Figure~\ref{fig.schubert_proof}, with
    $$\rho(um_{\ell}) = \begin{pmatrix}
        1-x \calP_{\ell-2} \calP_{\ell-1} + x^2 \calP_{\ell-1}^2 & * \\
        * & 1+x\calP_{\ell-2}\calP_{\ell-1}
    \end{pmatrix}
    $$
    and $\Tr \rho(um_{\ell}) = 2+ x^2 \calP_{\ell-1}^2$. As the homeomorphism between the orbifolds $\calO(p,q)$ and $\calO(p,\ell)$ is given by interchanging the role of $B_1$ and $B_2$, it swaps the loops $uv$ and $um_{\ell}$, and thus exchanges their traces $2+x^2$ and $2+x^2 \calP_{\ell-1}^2$. It follows that the map sending $x^\ast$ to $x \calP_{\ell-1}$ gives an isomorphism from $V^\ast$ to $V$.
    Interchanging the role of $B_1$ and $B_2$, we obtain its inverse, sending $x \in V$ to $x^\ast \calP_{q-1}^\ast(x^\ast)\in V^\ast$.
\end{proof}


\subsubsection{Cusp shape}
Let $m$ and $l$ be a meridian and a canonical longitude of $K(p,q)$, viewed as commuting elements in $G$ and let $\rho_{\text{geom}}:G\to \SL_2(\C)$ be a lift of the geometric representation, provided that $K(p,q)$ is hyperbolic. Up to conjugation, one can assume that 
$$\rho_{\text{geom}}(m)=\begin{pmatrix} 1 & 1 \\ 0 & 1 \end{pmatrix},\quad \rho_{\text{geom}}(l)=\begin{pmatrix} -1 & c \\ 0 & -1 \end{pmatrix}.$$
The element $c$ is called the cusp shape of $K(p,q)$: it is a geometric invariant of the knot, controlling the Euclidean structure at infinity.
By analogy, one can define the cusp shape associated to the tautological representation $\rho:G\to \SL_2(V)$. 
As the blackboard-framed longitude is given by $w^\ast w$ (recall the proof of Theroem~\ref{toradj}), it follows from Equations~\eqref{eqn.w} and \eqref{eqn.wstar} that
$$\rho(m)=\begin{pmatrix} 1 & x \\ 0 & 1\end{pmatrix},\quad \rho(l)=\begin{pmatrix}-1 & 2\iota \calQ_{p-2} - 2x \sum_{k=1}^{p-1} \epsilon_k\\ 0 & -1\end{pmatrix}.$$
Here the sum of the epsilons in $\rho(l)$ arises from the writhe correction.
We derive 
\begin{equation}
\label{eqn.cusp}
   c=\frac{2 \iota \calQ_{p-2}}{x} -2 \sum_{k=1}^{p-1} \epsilon_k \in V
\end{equation} which show a tight relation between $\calQ_{p-2}$ and the cusp shape of $K(p,q)$. Computing the same element for $K(p,\ell)$ yields 
$$\frac{\calQ_{p-2}(x)}{x}=\frac{\calQ^\ast_{p-2}(x^\ast)}{x^\ast}$$
where the isomorphism between $V$ and $V^\ast$ is used implicitly.
Note that the sum of the epsilons in~\eqref{eqn.cusp} does not change when we replace $q$ by $\ell$, since they are inverses of each other.
Substituting $x^\ast=x\calP_{\ell-1}(x)$ yields the following identity:
$$\calQ^\ast_{p-2}(x^\ast)=\calQ_{p-2}(x)\calP_{\ell-1}(x)\in V.$$
We verified this identity by computer for many values of $p$ and $q$.

\subsubsection{Invariance of the adjoint torsion}
Let us comment the corresponding equation for the adjoint torsion $\tau_2$. Anticipating sligthly the next section, define $\Omega$ so that we can write $$\tau_2=\frac{\Omega}{4x^2 \calP_{\ell-1}^2}=\frac{\Omega^\ast}{4(x^\ast)^2(\calP^\ast_{q-1})^2}$$
Here the second equality uses the isomorphism $V\simeq V^\ast$ implicitly, with the fact that $\ell^\ast=q$. Substituting $x=x^\ast \calP^\ast_{q-1}$, the reciprocal of the relation $x^\ast=x \calP_{\ell-1}$, we deduce that
\begin{equation}
    \label{eqn.Omega}
    \Omega^\ast(x^\ast)=\Omega(x)\calP_{\ell-1}(x)^{-2}
\end{equation}
This shows that the element $\Omega$ transforms non trivially under the isomorphism between $V$ and $V^\ast$, but in a similar way to $\calQ_{p-2}$.

\subsection{Reidemeister torsions in Frobenius algebra}
\label{sec.frob}

It is shown in \cite{S2P} that $V=\Q[x]/(P_{p-1})$ has a natural semi-simple Frobenius algebra structure, meaning that it is endowed with a linear form $\epsilon:V\to \Q$, a bilinear symmetric non-degenerate form $\eta:V\times V\to \Q$, and an invertible element $\Omega\in V$ related by the following equations:
\begin{equation}
    \label{eqn.str}
    \eta(x,y)=\epsilon(xy),\quad \epsilon(x)=\Tr_{V/\Q}(\Omega^{-1}x),\quad \Omega=\sum_{k} \frac{e_k^2}{\eta(e_k,e_k)}
\end{equation}
where $x$ and $y$ are any elements of $V$, and $e_k$ is an orthogonal basis for $\eta$.
In our case, the natural basis $\calP_0,\calP_1,\ldots, \calP_{p-2}$ of $V$ is orthogonal and satisfies $\epsilon(\calP_k)=0$ for $k>0$ and $\epsilon(\calP_0)=1$. Moreover, $\eta(\calP_k,\calP_k)=(-1)^k \epsilon_{k+1}$ so that 
$$\Omega=\sum_{k=0}^{p-2} (-1)^k \epsilon_{k+1} \calP_k^2.$$
Let us write $\Omega=\Omega^+-\Omega^-$ by decomposing the sum over even and odd indexes of the epsilons. Using Lemmas~\ref{lem.e1} and~\ref{lem.e2}, we have
$$\Omega^+=\sum_{k=1}^{(p-1)/2} -\epsilon_{2k} \calP_{2k-1}^2=\sum_{k=1}^{(p-1)/2}-\epsilon_{2k} \calP_{p-2k-1}^2=-\sum_{k=0}^{(p-3)/2}\epsilon_{p-2k-1} \calP_{2k}^2.$$
As $\epsilon_{p-k} = \epsilon_k$, it follows that $\Omega^+ = -\Omega^-$, and hence $\Omega^+ = \tfrac{1}{2}\Omega$.  
By Theorem~\ref{toradj}, we obtain
\begin{equation}\label{toromega}
    \tau_2 = \frac{\Omega^+}{2x^2 \calP_{\ell-1}^2}
           = \frac{\Omega}{4x^2 \calP_{\ell-1}^2}.
\end{equation}

This relation between $\tau_2$ and $\Omega$ implies that the inverse sum of $\tau_2$ vanishes; see Theorem~\ref{thm.invtau2} below. Note that a similar equation for non-parabolic representations was studied in \cite{Yoon}. See \cite{GKY} for the quantum-physical motivation behind such inverse sums of adjoint torsions.

\begin{Theorem}
\label{thm.invtau2}
    Fix coprime odd integers $p$ and $q$ satisfying $ 1 <q <p $, that is, fix a hyperbolic knot-bridge knot $K(p,q)$ and let $\calX$ be its parabolic character variety. Then
    \begin{equation*}
        \sum_{\rho \in \calX} \frac{1}{\tau_2(\rho)} = 0 \,.
    \end{equation*}
\end{Theorem}
\begin{proof}
    For any $f\in \Q[\calX]$, we have 
    \begin{equation*}
    \sum_{\rho\in \calX}f(\rho)=\Tr_{\Q[\calX]/\Q}(f)=\frac{1}{2}\Tr_{\Q[\hat{\calX}]/\Q}(f) \,.    
    \end{equation*}
    Therefore, it suffices to compute the trace of $1/\tau_2$ in $V=\Q[\hat \calX]$ over $\mathbb{Q}$. From Equations~\eqref{eqn.str} and \eqref{toromega}, one has
\begin{align*}
    \Tr_{V/\Q}(1/\tau_2) & = \Tr_{V/\Q} (4x^2 \calP_{\ell-1}^2 / \Omega) \\
     &= \epsilon (4x^2 \calP_{\ell-1}^2) \\
     & = 4\eta (\epsilon_\ell x \calP_{\ell-1},\epsilon_\ell x \calP_{\ell-1}) \\
     & = 4\eta  (\calP_{\ell} - \calP_{\ell-2}, \calP_{\ell} - \calP_{\ell-2})
\end{align*}
As $\calP_{k}$ is orthogonal with $\eta(\calP_k, \calP_k) = (-1)^k\epsilon_{k+1}$, one has
$$\eta  (\calP_{\ell} - \calP_{\ell-2}, \calP_{\ell} - \calP_{\ell-2})=(-1)^{\ell} ( \epsilon_{\ell+1} + \epsilon_{\ell-1})=0.$$ The last equality follows from the fact that $q \ell \equiv \pm1 \pmod{p}$, implying that $\epsilon_{\ell+1}$ and $\epsilon_{\ell-1}$ have different signs. This completes the proof.
\end{proof}

The above theorem motivates us to consider the inverse sum of $\tau_1$. Although we have no quantum-physical explanation for its simplicity—specifically, for its independence of $p$ and $q$—Theorem~\ref{thm.invtau1} below shows that the inverse sum of $\tau_1$ is $\pm 1$.

\begin{Theorem}
\label{thm.invtau1}
    Fix coprime odd integers $p$ and $q$ satisfying $1 <q <p $, that is, fix a hyperbolic knot-bridge knot $K(p,q)$ and let $\calX$ be its parabolic character variety. Then 
    \begin{equation*}
        \sum_{\rho \in \calX} \frac{1}{\tau_1(\rho)} = \epsilon_{(\ell'-1)/2}
    \end{equation*}
    where $0<\ell'<2p$ is the unique odd integer satisfying $\ell'q \equiv -1$ modulo $2p$.
\end{Theorem}
\begin{proof}
    It follows from Theorem~\ref{thm.simp} that  
    $$\Tr_{V/\Q}(1/\tau_1)=\frac{1}{2}\Tr_{V/\Q}(x^2 \calP_{\ell'-1})=\frac{1}{2}\Tr_{V/\Q}(x(\calP_{\ell'-2}-\calP_{\ell'})).$$ Recall that $\calP_0,\ldots,\calP_{p-2}$ form an orthogonal basis of $V$, and for $i$ and $j\in \Z/2p\Z$, $\eta(\calP_i,\calP_j)=0$ unless $i \equiv j$ or $i+j \equiv 2p-2$, in which case this scalar product gives $(-1)^i \epsilon_{i+1}$. 
    The trace of an endomorphism $M \in \text{End}(V)$ is given by $$\Tr_{V/\Q}(M)=\sum_{k=0}^{p-2}\eta(\calP_k, M\calP_k)\eta(\calP_k,\calP_k).$$ If we let $M$ be the multiplication by $x(\calP_{\ell'-2}-\calP_{\ell'})$, then  for $0\leq k \leq p-2$, 
    \begin{align*}
    M\calP_k&=(\calP_{\ell'-2}-\calP_{\ell'}) x\calP_k\\
    &=(\calP_{\ell'-2}-\calP_{\ell'})\epsilon_{k+1}(\calP_{k+1}-\calP_{k-1})\\
    &=\epsilon_{k+1}(-\calP_{\ell'+k+1}+\calP_{\ell'-k-3}+\calP_{\ell'+k-1}-\calP_{\ell'-k-1}).
    \end{align*}
    Here the last equality follows from Lemma~\ref{sumoftwo}.
    Therefore, one has
    \begin{align*}
        \Tr_{V/\Q}(M)&=\sum_{k=0}^{p-2}\epsilon_{k+1}\,\eta(\calP_k,\,-\calP_{\ell'+k+1}+\calP_{\ell'-k-3}+\calP_{\ell'+k-1}-\calP_{\ell'-k-1})\,\eta(\calP_k,\calP_k)\\
        &=\sum_{k=0}^{p-2}(-1)^k\,\eta(\calP_k,\,-\calP_{\ell'+k+1}+\calP_{\ell'-k-3}+\calP_{\ell'+k-1}-\calP_{\ell'-k-1}) \, .
    \end{align*}
    Expanding the bilinear form $\eta$ in the sum, we obtain four terms: $\eta(\calP_k, -\calP_{\ell'+k+1})$, $\eta(\calP_k,\calP_{\ell'-k-3})$, $\eta(\calP_k,\calP_{\ell'+k-1})$, and $\eta(\calP_k,\calP_{\ell'-k-1})$.
    We analyze the non-trivial contributions from each term individually. Recall that $\eta(\calP_k,\calP_i) \neq 0$ if and only if $k \equiv i$ or $k+i \equiv 2p-2 \pmod{2p}$.
    \begin{itemize}[leftmargin=*]
    \item First term: if $k \equiv \ell'+k+2 \pmod{2p}$, then $\ell'=-1$ which we excluded. The only possibility is to have $k+\ell'+k+1 \equiv 2p-2 \pmod{2p}$ which appears only for $k=(2p-\ell'-3)/2$. This contributes to the trace $-\epsilon_{k+1}=-\epsilon_{p-(\ell'+1)/2}=-\epsilon_{(\ell'+1)/2}$.
    \item Second term: if $k+\ell'-k-3\equiv 2p-2 \pmod{2p}$, then $\ell'$ is congruent to $1$ modulo $p$, which is impossible. If $k \equiv \ell'-k-3 \pmod{2p}$, then $k=(\ell'-3)/2$. This contributes $\epsilon_{(\ell'-1)/2}$.
    \item Third term: the case $k \equiv \ell'+k-1 \pmod{2p}$ is impossible. The other case  $k+\ell'+k-1\equiv 2p-2 \pmod{2p}$ happens when $k=p-(\ell'+1)/2$ and contributes $\epsilon_{p-(\ell'-1)/2}=\epsilon_{(\ell'-1)/2}$.
    \item  Fourth term: the case $k+\ell'-k-1 \equiv -2 \pmod{2p}$ is excluded. It remains the case $k \equiv \ell'-k-1 \pmod{2p}$ which happens when $k=(\ell'-1)/2$ and contributes $-\epsilon_{(\ell'+1)/2}$. 
    \end{itemize}
    Collecting the non-trivial terms, we conclude that the trace of $M$ is equal to $\epsilon_{(\ell'-1)/2}-2\epsilon_{(\ell'+1)/2}$. If we write $q\ell'=2ps-1$ for some integer $s$, it follows that $\epsilon_{(\ell'+1)/2}=(-1)^s$ and $\epsilon_{(\ell'-1)/2}=(-1)^{s+1}$, giving $\Tr_{V/\Q}(M)=4(-1)^{s+1}$. Combining all the above, one has
    \begin{equation*}
        \sum_{\rho \in \calX} \frac{1}{\tau_1(\rho)} = \frac{1}{2} \Tr_{V/\Q}(1/\tau_1) =
        \frac{1}{4}\Tr_{V/\Q}(x(\calP_{\ell'-2}-\calP_{\ell'})) = (-1)^{s+1} \,.
    \end{equation*} 
    This completes the proof.
\end{proof}

\begin{Remark}
For $q=1$, we have $\ell'=2p-1$ and $\calP_{\ell'-1}=\calP_{-2}=1$. In this case, a direct computation shows that $\Tr_{V/\Q}(1/\tau_1)=\Tr_{V/\Q}(x^2)/2=2-p$.
\end{Remark}

\section{Asymptotics of the signature function}\label{section:asympto}

\subsection{Statement of the theorem}

Let $\Gamma(2)$ be the principal congruence subgroup of level 2 in $\SL_2(\Z)$. Fix $M \in \Gamma(2)$ and define the sequence $(p_n,q_n)$ for odd $n$ by
\begin{equation}
\label{eqn.pnqn}
\begin{pmatrix}
    q_n\\
    p_n 
\end{pmatrix} = M 
\begin{pmatrix}
    1\\
    n
\end{pmatrix}  \,.
\end{equation}
Writing the matrix $M$ as
\begin{equation*}
    M=\begin{pmatrix} a & b \\ c & d \end{pmatrix}\in\Gamma(2) \,,
\end{equation*}
we have odd integers $a,d$ and even integers $b,c$ with $q_n = a+bn$ and $p_n = c+dn$.
As we want to study the asymptotics of signatures along this sequence $(p_n,q_n)$, we further impose that  $0\le b<d$. Then for $n$ big enough, $p_n$ and $q_n$ are coprime odd integers satisfying $0<q_n<p_n$.
We observe that $\lim_n\frac{q_n}{p_n}=\frac{b}{d}$ is always the quotient of an even number by an odd one. Moreover, up to a shift of $n$, the sequence $(p_n,q_n)$ only depends on $b/d$. With such a shift, we can moreover assume $c>0$. We denote by $\Gamma_S(2)$ be the subset of $\Gamma(2)$ consisting of matrices with $0 \leq b <d$ and $c>0$.

To state the asymptotic of the signature $\sigma_g(\frac{q_n}{p_n})$, we construct an explicit Frobenius algebra $W_M$ that depends on $M \in \Gamma_S(2)$ as follows. Set
\begin{equation}
\label{eqn.alpha}
    \alpha_r = 
    \begin{cases}
        \lfloor \frac{rc}{d}\rfloor-\lfloor \frac{(r-1)c}{d}\rfloor & \text{for } 0 < r <d \\[2pt]
        \lfloor \frac{rc}{d}\rfloor-\lfloor \frac{(r-1)c}{d}\rfloor-1 & \text{for } r= d
    \end{cases} \,.
\end{equation}
Note that $\alpha_r \geq 0$ and  $\alpha_1 + \cdots +\alpha_d=c-1$. We define
$$N(\alpha_r)=i^{\alpha_r}\begin{pmatrix} \alpha_r+1+x & i^{-1}(\alpha_r +x)\\ i^{-1}(\alpha_r +x) & -\alpha_r+1-x\end{pmatrix} \in \mathrm{GL}_2(\Z[i][x])$$
and set 
$$\begin{pmatrix} -iH_1 \\ H_2 \end{pmatrix}=N(\alpha_1)\overline{N(\alpha_2)}N(\alpha_3)\cdots \overline{N(\alpha_{d-1})}N(\alpha_d)\begin{pmatrix} 1 \\0 \end{pmatrix}\,.$$
From the fact that $\alpha_1 + \cdots +\alpha_d = c-1$ is odd, one checks that both $H_1$ and $H_2$ are integer polynomials of degree $d$.
Similarly, we define an integer polynomial $H_3$ of degree $b$ as 
$$\begin{pmatrix}
    H_3 \\ * 
\end{pmatrix}=N(\alpha_1)\overline{N(\alpha_2)}N(\alpha_3)\cdots \overline{N(\alpha_{b})}\begin{pmatrix}1 \\ 0\end{pmatrix}.$$
Provided that
\begin{equation}
\text{$H_1$ has no multiple roots and no common root with $H_3$,}
\tag{H}\label{cond:H}
\end{equation}
we define the semi-simple Frobenius algebra $W_M$ of degree $d$ by
\begin{equation*}
    W_M=\Q[x]/(H_1) \quad \text{with} \quad \Omega_{W_M}=-H_1'H_2 H_3^{-2}\,.
\end{equation*} 

\begin{Theorem}\label{mainsign}
Fix $M \in \Gamma_S(2)$ and let $(p_n,q_n)$ be the sequence defined in~\eqref{eqn.pnqn}. Then for any $g\ge 2$, 
\begin{enumerate}
    \item the signature $\sigma_g(\frac{q_n}{p_n})$ is a polynomial in $n$, and
    \item under the assumption \eqref{cond:H}, its asymptotic satisfies
$$\sigma_g \left(\frac{q_n}{p_n}\right)\underset{n\to\infty}{\sim}2\left(\frac{n^3}{2\pi^2}\right)^{g-1}\zeta(2g-2)\Tr_{W_M/\Q}\left(\Omega_{W_M}^{g-1}\right).$$
\end{enumerate}
\end{Theorem}

\subsection{Reciprocal Frobenius algebra}

Fix $M \in \Gamma_S(2)$ and $(p_n,q_n)$ as in Theorem~\ref{mainsign}. As $d q_n-b p_n=1$, $d$ is the inverse of $q_n$ modulo $p_n$. 
Let us denote by $V^*_n$ the Frobenius algebra associated to $(p_n,q_n)$ and $V_n$ the one associated to $(p_n,d)$. Here we adopt the notations from Section~\ref{sec.top}, except that the subscript $n$ is added, and we attach $\ast$ to the algebra associated with $(p_n, q_n)$. Recall from  Section~\ref{sec.dual} that they are isomorphic as algebras and denoting by $x^*,x$ their generators, one has the relations: 
$$x^*=x\calP_{q_n-1},\quad x=x^*\calP^*_{d-1}$$
together with (see Equation~\eqref{eqn.Omega})
$$\Omega^*_n=\Omega_n {\calP}_{q_n-1}^{-2}\,.$$
It follows that  
$$\sigma_g\left(\frac{q_n}{p_n}\right)=\Tr_{V^*_n/\Q} \left((\Omega^*_n)^{g-1}\right)=\Tr_{V_n/\Q} \left((\Omega_n \calP_{q_n-1}^{-2})^{g-1}\right).$$


\begin{Proposition}\label{description}
One has
$$\begin{pmatrix}P_{p_n-1}\\P_{p_n-2}\end{pmatrix}=\begin{pmatrix}(-1)^{d+1} X& 1 \\ 1 & 0\end{pmatrix}^{n+\alpha_d}\cdots\begin{pmatrix}-X& 1 \\ 1 & 0\end{pmatrix}^{n+\alpha_2}\begin{pmatrix} X& 1 \\ 1 & 0\end{pmatrix}^{n+\alpha_1}\begin{pmatrix} 1 \\ 0\end{pmatrix}$$
and
$$\begin{pmatrix}P_{q_n-1}\\P_{q_n-2}\end{pmatrix}=\begin{pmatrix}(-1)^{b+1} X& 1 \\ 1 & 0\end{pmatrix}^{n+\alpha_b}\cdots\begin{pmatrix}-X& 1 \\ 1 & 0\end{pmatrix}^{n+\alpha_2}\begin{pmatrix}X& 1 \\ 1 & 0\end{pmatrix}^{n+\alpha_1}\begin{pmatrix} 1 \\ 0\end{pmatrix}.$$
Here the integers $\alpha_r$, defined in~\eqref{eqn.alpha}, are independent of $n$.
\end{Proposition}
\begin{proof}
We first consider the $\epsilon$-sequence $\epsilon_k=(-1)^{\lfloor \frac{dk}{dn+c}\rfloor}$  for $1 \leq k \leq p_n -1$ associated to the Frobenius algebra $V_n$. It is a sequence of $d$ alternating blocks of $1$s and $-1$s. Precisely, to find the last term of the $r$-th block for $1 \leq r \leq d$, we compare $dk$ and $r(dn+c)$: 
writing $k=rn+s$, this amounts to compare $ds$ and $rc$. The last term of the $r$-th block is then $k_r=rn+\lfloor\frac{rc}{d}\rfloor$. In particular, the last term is $k_d=nd+c=p_n$. As we want to stop the sequence at $p_n-1$, we set instead $k_d=p_n-1$. We also set $k_0=0$ for notational convenience. Note that $k_r - k_{r-1} = n +\alpha_r$.

We now recall that $P_k=\epsilon_kXP_{k-1}+P_{k-2}$. Equivalently,
$$\begin{pmatrix}P_{k}\\P_{k-1}\end{pmatrix}=\begin{pmatrix}\epsilon_k X& 1 \\ 1 & 0\end{pmatrix}\begin{pmatrix} P_{k-1} \\ P_{k-2}\end{pmatrix}\,.$$
By the definition of $k_r$, one has $\epsilon_k=(-1)^r$ for all $k_{r-1} < k \leq k_{r}$. Since $k_r -k_{r-1} = n+\alpha_r$, we obtain the first formula.

On the other hand, as $ad-bc=1$, we have $\frac{bc}{d}=a-\frac{1}{d}$ and thus
$k_b=bn+\lfloor \frac{bc}{d}\rfloor = bn+a-1 =q_n -1$. It follows that $P_{q_n-1}$ is obtained from the same formula as $P_{p_n-1}$ by taking only the first $b$ blocks. This proves the proposition.
\end{proof}

Set $X= i(t+t^{-1})$ for a formal variable $t$. Then a direct computation shows that 
\begin{equation}
\label{eqn.power}
    \begin{pmatrix} X& 1 \\ 1 & 0\end{pmatrix}^k=\frac{i^k}{t-t^{-1}}\begin{pmatrix} t^{k+1}-t^{-k-1}& -i(t^{k}-t^{-k}) \\ -i(t^{k}-t^{-k}) & -(t^{k-1}-t^{-k+1})\end{pmatrix} .
\end{equation}
Changing the sign of $X$ amounts to changing $i$ to $-i$, hence
\begin{equation}
\label{eqn.power2}
    \begin{pmatrix}-X& 1 \\ 1 & 0\end{pmatrix}^k=\frac{i^{-k}}{t-t^{-1}}\begin{pmatrix} t^{k+1}-t^{-k-1}& i(t^{k}-t^{-k}) \\ i(t^{k}-t^{-k}) & -(t^{k-1}-t^{-k+1})\end{pmatrix} .
\end{equation}
Combining the above formulas with Proposition~\ref{description}, we obtain the following proposition, which produces three two-variable polynomials. We will need more precise information about these polynomials, which we defer to the technical Lemma~\ref{Qlemma}. 
\begin{Proposition}\label{changevariable}
There exist polynomials $Q_{M},R_{M},S_M\in \Z[U^{\pm 1},V^{\pm 1}]$, depending only on $M \in \Gamma_S(2)$, such that 
$$\begin{cases}P_{p_n-1}(it+it^{-1})(t-t^{-1})^d=i^{n+1}Q_{M}(t,t^n)\\
P_{p_n-2}(it+it^{-1})(t-t^{-1})^d=i^{n}R_{M}(t,t^n)\\
P_{q_n-1}(it+it^{-1})(t-t^{-1})^b=S_{M}(t,t^n)
\end{cases}$$
for any odd $n$.
\end{Proposition}
\begin{proof}
    It follows from Proposition~\ref{description} and Equations~\eqref{eqn.power}, \eqref{eqn.power2} that $P_{p_n-1}(it+it^{-1})$ can be expressed as a polynomial in $t$ and $t^n$ with coefficients in $\Z[i]$, multiplied by $i^{\sum_{r}(-1)^{r+1}(n+\alpha_r)}(t-t^{-1})^d=i^{n+\kappa}(t-t^{-1})^d$ where $\kappa=\sum_r (-1)^{r+1}\alpha_r$. As $\kappa\equiv\sum_r \alpha_r=c-1\equiv 1 \pmod{2}$, one gets $\pm i^{n+1}(t-t^{-1})^d$. Remaining powers of $i$ appear exactly in non-diagonal coefficients, and this property is stable by product.
%
    This proves the desired identity for $P_{p_n-1}$ and $P_{p_n-2}$. For $P_{q_n-1}$, there is no term of the form $i^n$ and the remaining power of $i$ is $\alpha_1-\alpha_2+\cdots+\alpha_b\equiv a-1\equiv 0$ modulo 2.
\end{proof}

\subsection{The signature is a polynomial in $n$}
Let us prove in this section the first item of Theorem \ref{mainsign}. 

From \cite{S2P}, one has $\Omega_n=-P_{p_n-2}P_{p_n-1}'$, hence 
\begin{equation}\label{signature-trace}
 \sigma_g  \left(\frac{q_n}{p_n} \right)=
\Tr_{V_n/\Q}\left( \big(-P_{q_n-1}^{-2} P_{p_n-2} P'_{p_n-1}\big)^{g-1}\right).
\end{equation}
Recall from the Residue Formula (see e.g. \cite{G}) that given any polynomial $P$ with simple roots and any $Q\in \Q[x]$ one has
\begin{equation*}
 \Tr_{\Q[x]/P}(Q)=\sum_{P(x)=0}Q(x)=\sum_{x \in \C}\Res_x(\frac{QP'}{P}dx)=-\Res_{x=\infty}(\frac{QP'}{P}dx).
\end{equation*}
We cannot apply this directly to the computation of $\sigma_g$ because we need to invert $P_{q_n-1}$, which creates new poles and invalidates the residue formula. To circumvent this problem, we use
$$ \frac{x^*}{x}=P_{q_n-1}(x),\quad \frac{x}{x^*}=P^*_{d-1}(x^*)$$
from which we deduce 
$$P_{q_n-1}(x)^{-1}=P^*_{d-1}(x^*)=P^*_{d-1}(xP_{q_n-1}(x)).$$
To simply the notation, let us set $Q(x)=P^*_{d-1}(xP_{q_n-1}(x))$.
In order to apply Proposition \ref{changevariable} , we perform the change of variables $x=it+it^{-1}$. Since this maps $t=0$ to $x=\infty$ and is non-ramified there, we can replace $x$ by $it+it^{-1}$, $x=\infty$ by $t=0$, and $dx$ by $\frac{i(t-t^{-1})dt}{t}$. It follows that $\sigma_g(\frac{q_n}{p_n})$ is equal to  
\begin{equation*}
\underset{t=0}{\Res}\left(
(-1)^g\frac{Q(it+it^{-1})^{2g-2}P_{p_n-2}(it+it^{-1})^{g-1} P'_{p_n-1}(it+it^{-1})^{g}}
{P_{p_n-1}(it+it^{-1})}\frac{i(t-t^{-1})dt}{t}
\right).
\end{equation*}
On the other hand, from Proposition~\ref{changevariable}, one computes
\begin{equation}\label{Pderiv}
P'_{p_n-1}(it+it^{-1})=i^{n}\left(\frac{t\frac{d}{dt}Q_M(t,t^n)}{(t-t^{-1})^{d+1}}-\frac{d(t+t^{-1})Q_M(t,t^n)}{(t-t^{-1})^{d+2}}\right).
\end{equation}
Substituting this into the above residue expression of $\sigma_g(\frac{q_n}{p_n})$, we obtain 
\begin{equation}
\label{eqn.sres}
\sigma_g \left( \frac{q_n}{p_n} \right) = \Res_{t=0} \left( \frac{R(t,t^n)}{Q_M(t,t^n)(t-t^{-1})^m} \dfrac{dt}{t}\right)
\end{equation}
for some integer $m$ independent of $n$ and some $R\in \Z[U^{\pm1},V^{\pm 1}]$. 

\begin{Lemma}
\label{lem.RS}
Let $R$ and $S\in \Z[U^{\pm 1},V^{\pm 1}]$ be two polynomials and suppose that the lowest $V$-coefficient of $S$ is $(U-U^{-1})^\ell$ modulo $\Z[U^{\pm 1}]^\times$. Then 
$$\Res_{t=0} \left(\frac{R(t,t^n)}{S(t,t^n)(t-t^{-1})^m}\frac{dt}{t}\right)$$
is a polynomial in $n$ of order less than $m+\ell$.
\end{Lemma}

\begin{proof}
It is sufficient to consider $R(U,V)=U^a V^b$ by the linearity of the residue. Let $d$ be the lowest $V$-degree of $S$. Then, by the assumption, we can write $$S(U,V)=\epsilon U^c(U-U^{-1})^\ell V^d(1-V\tilde{S}(U,V))$$ where $\epsilon=\pm 1$ and $\tilde{S}\in \Z[U^{\pm 1},V][[U]]$, hence 
\begin{align*}
&\frac{R(t,t^n)}{S(t,t^n)(t-t^{-1})^m}=\frac{t^{a+nb-c-nd+m+\ell}}{\epsilon(1-t^n\tilde{S}(t,t^n))(t^2-1)^{m+\ell}}\\
&=\epsilon^{-1} t^{a-c+m+\ell+n(b-d)}\sum_{i,j\ge 0}t^{ni}\tilde{S}(t,t^n)^it^{2j}(-1)^{m+\ell}\binom{j+m+\ell-1}{m+\ell-1}.
\end{align*}
We need to pick the constant coefficient in $t$: the key point is that only a finite number of $i's$ can contribute. Fix a value of $i$ and expand $\tilde{S}(t,t^n)^i$. All terms have the form 
$$C\sum_{j\ge 0}t^{A+Bn}t^{2j}\binom{j+m+\ell-1}{m+\ell-1}.$$
If $B<0$ and $A+Bn$ is even, the constant term of this series is $\binom{\frac{A+Bn}{2}+m+\ell-1}{m+\ell-1}$. This is a polynomial in $n$ (which is assumed to be odd) of degree $m+\ell-1$. This completes the proof.
\end{proof}

Lemma~\ref{Qlemma} below shows that Lemma~\ref{lem.RS} can be applied to Equation~\eqref{eqn.sres}. It follows that the signature $\sigma_g(\frac{q_n}{p_n})$ is a polynomial in $n$, which proves the first part of Theorem~\ref{mainsign}.

\subsection{The leading term of the signature}

We now prove the second part of Theorem \ref{mainsign}, starting again from Equation \eqref{signature-trace} that we rewrite in the following form:
$$\sigma_g \left(\frac{q_n}{p_n}\right)=\sum_{x\in Z_n}\left(\frac{-P_{p_n-2}(x)P'_{p_n-1}(x)}{P_{q_n-1}(x)^2}\right)^{g-1}$$
where $Z_n=\{x\in \C \,|\, P_{p_n-1}(x)=0\}$. From the relation between $P_{p_n-1}$ and $Q_M$ in Proposition~\ref{changevariable}, one sees a direct relationship between the zeros of $P_{p_n-1}$ and the zeros of $Q_M$. To be more precise, we define 
$$\calC=\{(u,v)\in (\C^*)^2 \,|\, Q_M(u,v)=0\},\quad \calC[n]=\{(u,v)\in \calC\setminus \{\pm 1\}^2\,|\, v=u^n\}.$$
The map $\calC[n]\to Z_n$ given by $(u,v)\mapsto i(u+u^{-1})$ is surjective and $2$-to-$1$: two preimages are related by the involution $(u,v)\mapsto(u^{-1},v^{-1})$.

\begin{Lemma}
The set $\bigcup_{n\ge 1}\calC[n]$ is relatively compact in $\calC$. 
\end{Lemma}
\begin{proof}
This amounts to showing that points in $\calC[n]$ cannot accumulate to ideal points of $\calC$. Recall the well-known fact that such ideal points correspond to zeroes of the side polynomials of $Q_M$, that is, the one-variable polynomials associated with  the sides of the Newton polygon of $Q_M$. To see this, fix an ideal point of $\calC$ and take Puiseux coordinates of the form $u=s^p$ and $v=s^qF(s)$, where $p,q$ are coprime and $F$ is a convergent series with $F(0)\neq 0$. As $Q_M(s^p,s^qF(s))=0$ identically, degree considerations imply that the linear form $(i,j)\mapsto pi+qj$ should support a face of the Newton polygon of $Q_M$. 

By Lemma \ref{Qlemma} below, this polygon is an explicit parallelogram, so the only possibilities are $(p,q)\in\{(0,1),(0,-1),(d,-c),(-d,c)\}$. In the first two cases, the Puiseux coordinates have the form $u=s^{\pm 1}$ and $v=F(s)$. As $F(0)\ne 0$, it is impossible to have $F(s)=s^{\pm n}$ in a neighborhood of $0$. In the latter two cases, we have $u=s^{\pm d}$ and $v=s^{\mp c}F(s)$. Again, the equation $v=u^n$ forces $F(s)=s^{ \pm nd\pm c}$, which is impossible for $s$ sufficiently close to $0$. This completes the proof.
\end{proof}

Recall the formula 
$P_{p_n-1}(it+it^{-1})(t-t^{-1})^d=i^{n+1} Q_M(t,t^n)$ of Proposition~\ref{changevariable}. Differentiating it at $(t,t^n)\in \calC[n]$, one has
\begin{equation}\label{derivP}
P_{p_n-1}'(it+it^{-1})i(t-t^{-1})^{d+1}=i^{n+1}t\partial_UQ_M(t,t^n)+i^{n+1}nt^{n}\partial_VQ_M(t,t^n).
\end{equation}
Hence, setting $$F_n(u,v)=(u-u^{-1})^{2b-2d-1}\frac{R_M(u,v)}{S_M(u,v)^2}\left( u\partial_UQ_M(u,v)+nv\partial_VQ_M(u,v)\right)$$
we can write
$$\sigma_g \left(\frac{q_n}{p_n}\right)=\frac{1}{2}\!\!\!\sum_{(u,v)\in \calC[n]}F_n(u,v)^{g-1}.$$

\begin{Lemma}
\label{lem.bdd1}
For any neighborhood $W$ of $\{\pm 1\}^2\subset \calC$, there is a constant $C>0$ such that $|F_n(u,v)|\le C n$ for any $(u,v)\in \calC[n]\setminus W$. 
\end{Lemma}
 
\begin{proof}
Let us write $x=it+it^{-1}$. As $P_{p_n-2}(x)^2=-1$  for any $x\in Z_n$ (see Proposition~\ref{prop.iota}), this term can readily be neglected. 
As $(t,t^n)$ belongs to a compact subset of $\calC$, the polynomials $U\partial_UQ_M(U,V)$ and $V\partial_VQ_M(U,V)$ are uniformly bounded on $\calC[n]$. Moreover, by the third point of Lemma \ref{Qlemma}, the set of $(u,v)\in \calC$ with $u=\pm 1$ is reduced to $\{\pm 1\}^2$. This means that there exists $\epsilon>0$ such that for all $(u,v)\in \calC\setminus W$ one has $|u^2-1|>\epsilon$. Combining these with Equation~\eqref{derivP}, we deduce that there is a constant $C>0$ such that $|P'(iu+iu^{-1})|\le Cn$ for all $(u,v)\in \calC[n]\setminus W$. 

It remains to deal with $P_{q_n-1}(x)^{-2}$. If we first replace it with $P^*_{d-1}(xP_{q_n-1}(x))^2$, we are reduced to showing that $P_{q_n-1}(x)$ is uniformly bounded.
To this aim, recall from Proposition \ref{changevariable} that $P_{q_n-1}(x)(t-t^{-1})^b=S_M(t,t^n)$. The same argument as for $P'_{q_n-1}$ shows that $P_{q_n-1}(it+it^{-1})$ is uniformly bounded for $(t,t^n)\in \calC[n]\setminus W$, proving the lemma. 
\end{proof}


Since the cardinality of $\calC[n]$ is linear in $n$, from the above lemma, one has $$\sum_{(u,v)\in \calC[n] \setminus W} \hspace{-0.2cm} F_n(u,v)^{g-1}=O(n^g)$$ which is negligible compared to the expected $n^{3g-3}$ for $g\ge 2$. Therefore, henceforth we concentrate on $\calC[n]\cap W$.

Let $W^{\pm }$ be a neighborhood $(1,\pm 1)\in \calC$. For parity reasons, the neighborhood of $(-1,\pm 1)$ will contribute in the same way to the sum. The previous lemma implies that 
$$\sigma_g \left(\frac{q_n}{p_n}\right)= \sum_{(u,v)\in \calC[n]\cap (W^+\cup W^-)} \hspace{-0.7cm} F_n(u,v)^{g-1}+O(n^g)$$
Let $\overline{\calC}$ be the normalization of $\calC$ and suppose that the first projection $\overline{\calC}\to\C$ is not ramified over $1$. This implies that there exist $d-1$ branches of $\calC$ with Puiseux coordinates of the form $u=e^s, v=e^{F_j(s)}$ with $F_j(0)=0$ and $F_j'(0)\ne 0$. By the symmetries of $Q_M$ given in Lemma \ref{Qlemma}, the branches over $(\epsilon_u,\epsilon_v)$ are parametrized by $u=\epsilon_u e^s,v=\epsilon_v e^{F_j(s)}$ for $\epsilon_u,\epsilon_v\in \{\pm 1\}$. 

Recall that an element of $\calC[n]$ in the $j$-th branch satisfies $\pm e^{F_j(s)}=\pm e^{ns}$. Equivalently, $F_j(s)=ns-i\pi k$ for some $k\in \Z$. As $F_j(s)$ is close to $0$, there is a unique value of $k\ne 0$ which can solve this equation, reciprocally given $k\ne 0$, there is a unique solution $s_j^k(n)$ of this equation that we can expand into powers of $\frac{1}{n}$ in the form 
$$s_j^k(n)=\frac{i\pi k}{n}\Big(1+F_j'(0)\frac{1}{n}+O(\frac{1}{n^2})\Big).$$
Along such a sequence, we have $u=e^{\frac{i\pi k}{n}+O(\frac{1}{n^2})}$ and $v=(-1)^ke^{\frac{i\pi k }{n}F_j'(0)+O(\frac{1}{n^2})}$. Depending on the parity of $k$, this will converge to $(1,1)$ or $(1,-1)$. 
By the fourth point of Lemma \ref{Qlemma}, one sees that $F_j'(0)$ is a root of $H_1$. As $H_1$ has simple roots by assumption, it justifies our assumption on the ramification of $\calC$ over $1$. 

One can reduce to the case $k>0$: the case $k<0$ corresponds to the involution $(u,v)\mapsto (u^{-1},v^{-1})$ and gives the same contribution. 

Let us compute the behavior of $F_n(u,v)\sim n(u-u^{-1})^{2b-2d-1}\frac{R_M(u,v)v\partial_VQ_M(u,v)}{S_M(u,v)^2}$ on this sequence. 
\begin{itemize}
\item[-] Directly, $(u-u^{-1})^{2b-2d-1}\sim (\frac{2i\pi k}{n})^{2b-2d-1}$.
\item[-] By Lemma \ref{Qlemma} point (6): $S_M(u,v)^2\sim 2^{2b} (\frac{i\pi k}{n})^{2b} H_3(F_j'(0))^2$.
\item[-] By Lemma \ref{Qlemma} point (5): $R_M(u,v)\sim 2^d (\frac{i\pi k}{n})^dH_2(F_j'(0))$.
\item[-] By Lemma \ref{Qlemma} point (4): $v\partial_V Q_M(u,v)\sim 2^d(\frac{i\pi k}{n})^{d-1}H_1'(F_j'(0))$.
\end{itemize}
Putting everything together gives 
$$F_n(e^{s_j^k(n)},e^{ns_j^k(n)})\sim -\frac{n^3}{2(\pi k)^2}\frac{H_1'(F_j'(0))H_2(F_j'(0))}{H_3(F_j'(0))^2}.$$
Summing over $k>0$ and over $j$ gives the final formula
$$\sigma_g \left(\frac{q_n}{p_n}\right)\sim 2\left(\frac{n^3}{2\pi^2}\right)^{g-1} \!\!\zeta(2g-2)\sum_{j}\left(-\frac{H_1'(F_j'(0))H_2(F_j'(0))}{H_3(F_j'(0))^2}\right)^{g-1} .$$

\subsection{The main lemma}
Here we state and prove the main lemma used in the last two subsections. 
\begin{Lemma}\label{Qlemma}
The polynomial $Q_M$ satisfies the symmetries 
$$Q_M(-U,V)=Q_M(U,V)=-Q_M(U,-V)=-Q_M(U^{-1},V^{-1})$$ 
together with the following properties: 
\begin{enumerate}
\item The terms with highest (resp. lowest) power of $V$ are $\pm V^d(U+U^{-1})^{d-1}U^{c}$ and $\pm V^{-d}(U+U^{-1})^{d-1}U^{-c}$ respectively.
\item Its Newton polygon is the parallelogram of vertices $\pm (d-1,0)\pm(c,d)$.
\item $Q_M(\pm 1,V)=\pm 2^{d-1}(V-V^{-1})^d$.
\item $Q_M(e^u,e^v)=2^du^dH_1(v/u)$ modulo terms in $(u,v)$ of total order $<d$. 
\end{enumerate}
Similarly, the polynomials $R_M$ and $S_M$ satisfy
\begin{enumerate}
\item[(5)] $R_M(e^u,e^v)=2^du^dH_2(v/u)$  modulo terms in $(u,v)$ of total order $<d$. 
\item[(6)] $S_M(e^u,e^v)=2^bu^bH_3(v/u)$ modulo terms in $(u,v)$ of total order $<b$.  
\end{enumerate}
\end{Lemma}
\begin{proof}
Before all, we provide an explicit formula for $Q_M$. Set $X=it+it^{-1},U=t,V=t^n$, and $N_k=\begin{pmatrix} i(t+t^{-1}) & 1 \\ 1 & 0 \end{pmatrix}^k(t-t^{-1})$. We compute
\begin{eqnarray*}
N_{n+\alpha}&=&\begin{pmatrix}i^{n+\alpha}t^{n+\alpha+1}-i^{n+\alpha}t^{-n-\alpha-1}& i^{n+\alpha-1}t^{n+\alpha}-i^{n+\alpha-1}t^{-n-\alpha} \\ i^{n+\alpha-1}t^{n+\alpha}-i^{n+\alpha-1}t^{-n-\alpha} & i^{n+\alpha-2}t^{n+\alpha-1} -i^{n+\alpha-2}t^{-n-\alpha+1}\end{pmatrix}\\
&=&i^{n+\alpha}\begin{pmatrix}U^{\alpha+1}V-U^{-\alpha-1}V^{-1}& i^{-1}U^{\alpha}V-i^{-1}U^{-\alpha}V^{-1}\\ i^{-1}U^{\alpha}V-i^{-1}U^{-\alpha}V^{-1}& -U^{\alpha-1}V+U^{-\alpha+1}V^{-1}\end{pmatrix}
\end{eqnarray*}
If we index the elements of these matrices by signs, say $\eta,\xi\in \{\pm 1\}$, the entry $\eta,\xi$ of the matrices above are given by 
$$ \left(N_{n+\alpha}\right)_{\eta,\xi}=i^{n+\alpha-1+\frac{\eta+\xi}{2}}\Big(U^{\alpha+\frac{\eta+\xi}{2}}V-U^{-\alpha-\frac{\eta+\xi}{2}}V^{-1}\Big).$$
It follows that setting $\calN=N_{n+\alpha_1}\overline{N_{n+\alpha_{2}}}\cdots \overline{N_{n+\alpha_{d-1}}}N_{n+\alpha_d}$, its entry $\calN_{\xi_0,\xi_d}$ equals:
$$\sum_{\xi_1,\ldots,\xi_{d-1}} i^{\sum_{r=1}^d (-1)^{r+1}(n+\alpha_r-1+\frac{\xi_{r-1}+\xi_r}{2})}\prod_{r=1}^d\Big(U^{\alpha_r+\frac{\xi_{r-1}+\xi_r}{2}}V-
U^{-(\alpha_r+\frac{\xi_{r-1}+\xi_r}{2})}V^{-1}\Big)$$
We compute that the power of $i$ reduces to $n+\kappa-1+\frac{\xi_d+\xi_0}{2}$ where $\kappa=\sum (-1)^{r+1}\alpha_r$ and the product can further be expanded as follows:
$$\sum_{\xi_1,\ldots,\xi_{d-1}\in\{\pm 1\}}\sum_{\eta_1,\ldots,\eta_d\in\{\pm1\}}\Big(\prod_{r=1}^d\eta_r\Big) U^{\sum_{r=1}^d \eta_r (\alpha_r+\frac{\xi_{r-1}+\xi_r}{2})}V^{\sum_{r=1}^d \eta_r}.$$
Finally, we sum over $\xi_1,\ldots,\xi_{d-1}$ to obtain
\begin{equation}\label{PolQ}
\sum_{\eta_1,\ldots,\eta_d\in\{\pm1\}}\Big(\prod_{r=1}^d\eta_r\Big) U^{\sum_{r=1}^d \eta_r \alpha_r+\frac{\eta_1\xi_{0}+\eta_d\xi_d}{2}}\prod_{r=1}^{d-1}\Big(U^{\frac{\eta_r+\eta_{r+1}}{2}}+U^{-\frac{\eta_r+\eta_{r+1}}{2}}\Big)V^{\sum_{r=1}^d \eta_r}.
\end{equation}
To recover $Q_M(U,V)$ from this formula, we simply take $\xi_0=\xi_d=+1$ and multiply the result by the sign $i^{\kappa-1}$. It follows readily from this expression that $Q_M$ contains only even powers of $U$ and odd powers of $V$. The symmetry $Q_M(U^{-1},V^{-1})=-Q_M(U,V)$ is clear from the change $\eta_r\mapsto -\eta_r$. 

From this formula, one sees that the maximal power of $V$ in $Q_M$ is $d$ and corresponds to putting $\eta_r=1$ for all $r$: its coefficient is 
$$i^{\kappa-1} U^{\sum \alpha_r+1}(U+U^{-1})^{d-1}=i^{\kappa-1} U^{c}(U+U^{-1})^{d-1}.$$
The case of the minimal coefficient is similar and gives $i^{\kappa-1}U^{-c}(U+U^{-1})^{d-1}$.

Let us prove the second statement, we consider the linear form $L:\Z^2\to\Q$ given by $L(i,j)=i-\frac{c}{d}j$. It takes the values $\pm (d-1)$ on the sides of the parallelogram expected as the Newton polygon of $Q_M$. Hence one needs to prove that any non trivial monomial $U^iV^j$ appearing in $Q_M$ satisfies $|L(i,j)|\le d-1$. It reduces to taking a sequence $\eta_1,\ldots,\eta_d\in \{\pm 1\}$ and compute 
\begin{align*}
\begin{split}
L \Big(\sum_{r=1}^d \eta_r \alpha_r&+\frac{\eta_1+\eta_d}{2}+\sum_{r=1}^{d-1}\pm\frac{\eta_r+\eta_{r-1}}{2},\ \sum_{r=1}^d \eta_r\Big)\\
&=\sum_{r=1}^{d}\eta_r\alpha_r+\frac{\eta_1+\eta_d}{2}+\sum_{r=1}^{d-1} \pm \frac{\eta_r+\eta_{r+1}}{2}-(\sum_{r=1}^d \eta_r)\frac{c}{d}.
\end{split}
\end{align*}

Let $S$ be the set of indices $0< r<d$ for which $\eta_r\ne \eta_{r+1}$: one has 
$$\Big|\sum_{r=1}^{d-1} \pm \frac{\eta_r+\eta_{r+1}}{2}\Big|\le d-1-\#S$$
so that we are reduced to showing $|\sum \eta_r\alpha_r+\frac{\eta_1+\eta_d}{2}-(\sum\eta_r)\frac{c}{d}|\le \#S$. We do a discrete integration by parts in each sum.

On one hand: $\sum_{r=1}^d\eta_r\alpha_r=\sum_{r=1}^d\eta_r(\beta_r-\beta_{r-1})=\sum_{r=0}^{d-1}(\eta_r-\eta_{r+1})\beta_r+\eta_d(c-1)$ where $\beta_r=\lfloor \frac{rc}{d}\rfloor$ for $0\le r<d$ and $\beta_d=c-1$. 
The sequence $\eta_r-\eta_{r+1}$ vanishes if $r\notin S$ so that this sum reads $2\sum_{r\in S}\eta_r\beta_r+\eta_d(c-1)$.

On the other hand: $\sum_{r=1}^d \eta_r=\sum_{r=1}^{d-1}r(\eta_r-\eta_{r+1})+d\eta_d=2\sum_{r\in S}r\eta_r+d\eta_d$. We are then reduced to 
$$\Big|2\sum_{r\in S}\eta_r(\beta_r-r\frac{c}{d})+\frac{\eta_1-\eta_d}{2}\Big|\le \#S.$$
As $-1\le \beta_r-\frac{rc}{d}\le 0$ and the sequence $\eta_r$ is alternating, the sum of two consecutive terms does not exceed $1$ in absolute value. If $\eta_1=\eta_d$, $\#S$ is even and the result follows. When $\eta_1\ne\eta_d$ then $\#S$ is odd: either the first or the last term of the sum has a sign opposite to $\frac{\eta_1-\eta_d}{2}$: again their sum does not exceed 1 in absolute value and the argument repeats and the second point is proved. 
We also observe that in case of equality, the sequence of $\xi$s and the sequence of $\eta$s should be constant: this shows that the non trivial coefficient on the slanted sides occur only at vertices. 

Let us prove the third statement, we simply put $U=1$ in Equation \eqref{PolQ}. We also observe from the expression just above that powers of $U$ in $Q_M$ have the form $\sum_{r=1}^d \eta_r(\alpha_r+\frac{\xi_{r-1}+\xi_{r}}{2})$ which is congruent modulo 2 to $\sum_{r=1}^{d}\alpha_r+\frac{\xi_0+\xi_d}{2}=c$. Hence $Q$ has only even powers of $U$ and the statement follows. 

Finally the last three statements are clear once we have observed that writing $U=e^u, V=e^v$ one has 
\begin{align*}
i^{n+\alpha}&\begin{pmatrix}U^{\alpha+1}V-U^{-\alpha-1}V^{-1}& i^{-1}U^{\alpha}V-i^{-1}U^{-\alpha}V^{-1}\\ i^{-1}U^{\alpha}V-i^{-1}U^{-\alpha}V^{-1}& -U^{\alpha-1}V+U^{-\alpha+1}V^{-1}\end{pmatrix}\\
&=i^{n+\alpha}\begin{pmatrix} 2\sinh((\alpha+1)u+v) & \frac{2}{i}\sinh(\alpha u+v)\\ \frac{2}{i}\sinh(\alpha u+v) & -2\sinh((\alpha-1)u+v)\end{pmatrix}\\
&=2i^{n+\alpha}\begin{pmatrix} (\alpha+1)u+v & i^{-1}(\alpha u+v)\\ i^{-1}(\alpha u+v) & -(\alpha-1)u-v\end{pmatrix}+\text{higher order terms}.
\end{align*}
Taking a product of $d$ such matrices, the first order will be given by the product of the first orders. The formulas $(4,5,6)$ are direct consequences of the definition of $H_1,H_2,H_3$.  
\end{proof}

{
\footnotesize
\bigskip
\noindent \textsc{D\'epartement de math\'ematiques et applications,
Ecole Normale Sup\'erieure, Université PSL, CNRS, Sorbonne Université,
75005 Paris, France}\\
\url{https://webusers.imj-prg.fr/~julien.marche/}\\
Email address: \texttt{julien.marche@ens.psl.eu}
\bigskip

\noindent \textsc{Department of Mathematics, Chonnam National University, Gwangju, South Korea}\\
\url{http://sites.google.com/view/seokbeom}\\
Email address: \texttt{sbyoon15@gmail.com}
}
\end{document}